\documentclass{article}

\usepackage[margin=1in]{geometry}
\usepackage{amsmath, amsthm, amssymb}
\usepackage{dsfont} \newcommand{\mathbbm}{\mathds}
\usepackage{mathrsfs}
\usepackage{hyperref}
\usepackage{cleveref}
\usepackage{tikz}
\usepackage{graphicx}

\newcommand{\ZZ}{\mathbb{Z}}
\newcommand{\RR}{\mathbb{R}}

\newcommand{\SL}{\operatorname{SL}}
\newcommand{\BCZ}{\operatorname{BCZ}}
\newcommand{\FTR}{\operatorname{FTR}}
\newcommand{\slope}{\operatorname{slope}}
\newcommand{\area}{\operatorname{area}}
\newcommand{\Unif}{\operatorname{Unif}}
\newcommand{\FareyStat}{\operatorname{FareyStat}}
\newcommand{\slopegap}
{\operatorname{slopegap}}
\newcommand{\SlopeGap}
{\operatorname{SlopeGap}}
\newcommand{\centdist}{\operatorname{centdist}}
\newcommand{\CentDist}{\operatorname{CentDist}}

\newtheorem{definition}{Definition}[section]
\newtheorem{remark}{Remark}[section]
\newtheorem{theorem}{Theorem}[section]
\newtheorem{proposition}{Proposition}[section]
\newtheorem{corollary}{Corollary}[section]
\newtheorem{lemma}{Lemma}[section]

\title{The Boca-Cobeli-Zaharescu Map Analogue for the Hecke Triangle Groups $G_q$}
\author{Diaaeldin Taha \\ University of Washintgon in Seattle, \href{mailto:dtaha@uw.edu}{dtaha@uw.edu}}

\begin{document}
\maketitle

\abstract{The Farey sequence $\mathcal{F}(Q)$ at level $Q$ is the sequence of irreducible fractions in $[0, 1]$ with denominators not exceeding $Q$, arranged in increasing order of magnitude. A simple ``next-term'' algorithm exists for generating the elements of $\mathcal{F}(Q)$ in increasing or decreasing order. That algorithm, along with a number of other properties of the Farey sequence, was encoded by F. Boca, C. Cobeli, and A. Zaharescu into what is now known as the Boca-Cobeli-Zaharescu (BCZ) map, and used to attack several problems that can be described using the statistics of subsets of the Farey sequence. In this paper, we derive the Boca-Cobeli-Zaharescu map analogue for the discrete orbits $\Lambda_q = G_q(1, 0)^T$ of the linear action of the Hecke triangle groups $G_q$ on the plane $\RR^2$ starting with a Stern-Brocot tree analogue for the said orbits (\cref{theorem: G_q BCZ maps}). We derive the next-term algorithm for generating the elements of $\Lambda_q$ in vertical strips in increasing order of slope, and present a number of applications to the statistics of $\Lambda_q$.}

\section{Introduction}

For any integer $Q \geq 1$, the \emph{Farey sequence at level $Q$} is the set
\begin{equation*}
\mathcal{F}(Q) = \{a/q \mid a,q \in \ZZ,\ 0 \leq a \leq q \leq Q,\ \gcd(a, q) = 1\}
\end{equation*}
of irreducible fractions in the interval $[0, 1]$ with denominators not exceeding $Q$, arranged in increasing order. The Farey sequence is one of the famous enumerations of the rationals, and its applications permeate mathematics. Some of the fundamental properties of $\mathcal{F}(Q)$ are the following:
\begin{enumerate}
\item If $a_1/q_1, a_2/q_2 \in \mathcal{F}(Q)$ are two consecutive fractions, then $0 < q_1, q_2 \leq Q$, and $q_1 + q_2 \geq Q$.
\item If $a_1/q_1, a_2/q_2 \in \mathcal{F}(Q)$ are two consecutive fractions, then they satisfy the \emph{Farey neighbor} identity $a_2 q_1 - a_1 q_2 = 1$.
\item If $a_1/q_1, a_2/q_2, a_3/q_3 \in \mathcal{F}(Q)$ are three consecutive fractions, then they satisfy the \emph{next-term} identities
\begin{equation*}
a_3 = k a_2 - a_1,
\end{equation*}
and
\begin{equation*}
q_3 = k q_2 - q_1
\end{equation*}
where $k = \left\lfloor\frac{Q + q_1}{q_2}\right\rfloor$.
\end{enumerate}
Around the turn of the new millennium, F. P. Boca, C. Cobeli, and A. Zaharescu \cite{Boca2001-he} encoded the above properties of the Farey sequence as the \emph{Farey triangle}
\begin{equation*}
\mathscr{T} := \{(a, b) \mid 0 < a, b \leq 1,\ a + b > 1\},
\end{equation*}
and what is now increasingly known as the Boca-Cobeli-Zaharescu (BCZ) map $T : \mathscr{T} \to \mathscr{T}$\footnote{In the remainder of this paper, we will denote the BCZ maps we compute for the Hecke triangle groups $G_q$ by $\BCZ_q$, reserving the symbols $T_q$ for particular generators of $G_q$.}
\begin{equation*}
T(a, b) := \left(b, -a + \left\lfloor\frac{1 + a}{b}\right\rfloor b \right)
\end{equation*}
which satisfies the property that
\begin{equation*}
T\left(\frac{q_1}{Q}, \frac{q_2}{Q}\right) = \left(\frac{q_2}{Q}, \frac{q_3}{Q}\right)
\end{equation*}
for any three consecutive fractions $a_1/q_1, a_2/q_2, a_3/q_3 \in \mathcal{F}(Q)$. Since then, quoting R. R. Hall, and P. Shiu \cite{Hall2003-it}, the aforementioned trio have ``made some very interesting applications'' of $\mathscr{T}$ and $T$ (and the weak convergence of particular measures on $\mathscr{T}$ supported on the orbits of $T$ to the Lebesgue probability measure $dm = 2 da db$) to the study of distributions related to Farey fractions.

Earlier in the current decade, J. Athreya, and Y. Cheung \cite{Athreya2013-ql} showed that the Farey triangle $\mathscr{T}$, the BCZ map $T : \mathscr{T} \to \mathscr{T}$, and the Lebesgue probability measure $dm = 2\,da\,db$ on $\mathscr{T}$ form a Poincar\'{e} section with roof function $R(a,b) = \frac{1}{ab}$ to the horocycle flow $h_s = \begin{pmatrix}1 & 0\\-s & 1\end{pmatrix}$, $s \in \RR$, on $X_2 = \SL(2,\RR)/\SL(2,\ZZ)$, with (a scalar multiple of) the Haar probability measure $\mu_2$ inherited from $\SL(2,\RR)$. Following that, analogues of the BCZ map have been computed for the golden L translation surface (whose $\SL(2, \RR)$ orbit corresponds to $\SL(2, \RR)/G_5$, where $G_5$ is the Hecke triangle group $(2, 5, \infty)$) by J. Athreya, J. Chaika, and S. Lelievre\footnote{By \cref{theorem: G_q BCZ maps}, we get for $q = 5$ the indices $k_2(a, b) = \left\lfloor\frac{1-(a+\varphi b)}{\varphi^2(a + b)}\right\rfloor$, $k_3(a, b) = \left\lfloor \frac{1-b}{\varphi(a+\varphi b)} \right\rfloor$, and $k_4(a, b) = \left\lfloor \frac{1+a}{\varphi b} \right\rfloor$. This corrects the indices given in theorem 3.1 of \cite{Athreya2015-nq}.} in \cite{Athreya2015-nq}, and later for the regular octagon by C. Uyanick, and G. Work in \cite{Uyanik2016-gx}. In both cases, the sought for application was determining the slope gap distributions for the holonomy vectors of the golden L and the regular octagon. Soon after, B. Heersink \cite{Heersink2016-hg} computed the BCZ map analogues for finite covers of $\SL(2, \RR)/\SL(2, \ZZ)$ using a process developed by A. M. Fisher, and T. A. Schmidt \cite{Fisher2014-ke} for lifting Poincar\'{e} sections of the geodesic flow on $\SL(2, \RR)/SL(2, \ZZ)$ to covers of thereof. In that case, the sought for application was studying statistics of various subsets of the Farey sequence.

In this paper, we derive the BCZ map analogue for the Hecke triangle groups $G_q$, $q \geq 3$, which are the subgroups of $\SL(2, \RR)$ with generators
\begin{equation*}
S := \begin{pmatrix}0 & -1 \\ 1 & 0\end{pmatrix}, \text{ and } T_q := \begin{pmatrix}1 & \lambda_q \\ 0 & 1\end{pmatrix},
\end{equation*}
where $\lambda_q := 2 \cos\left(\frac{\pi}{q}\right) \geq 1$. Along the way, we investigate the discrete orbits
\begin{equation*}
\Lambda_q = G_q(1, 0)^T
\end{equation*}
of the linear action of $G_q$ on the plane $\RR^2$, and present some results on the geometry of numbers, Diophantine properties, and statistics of $\Lambda_q$. Our starting point is showing that the orbits $\Lambda_q$ have a tree structure that extend the famous Stern-Brocot trees for the rationals. The said trees were studied a bit earlier by C. L. Lang and M. L. Lang in \cite{Lang2016-qs}, though their focus was on the M\"{o}bius action of $G_q$ on the hyperbolic plane.

\begin{figure}
\centering
\includegraphics[scale=0.7]{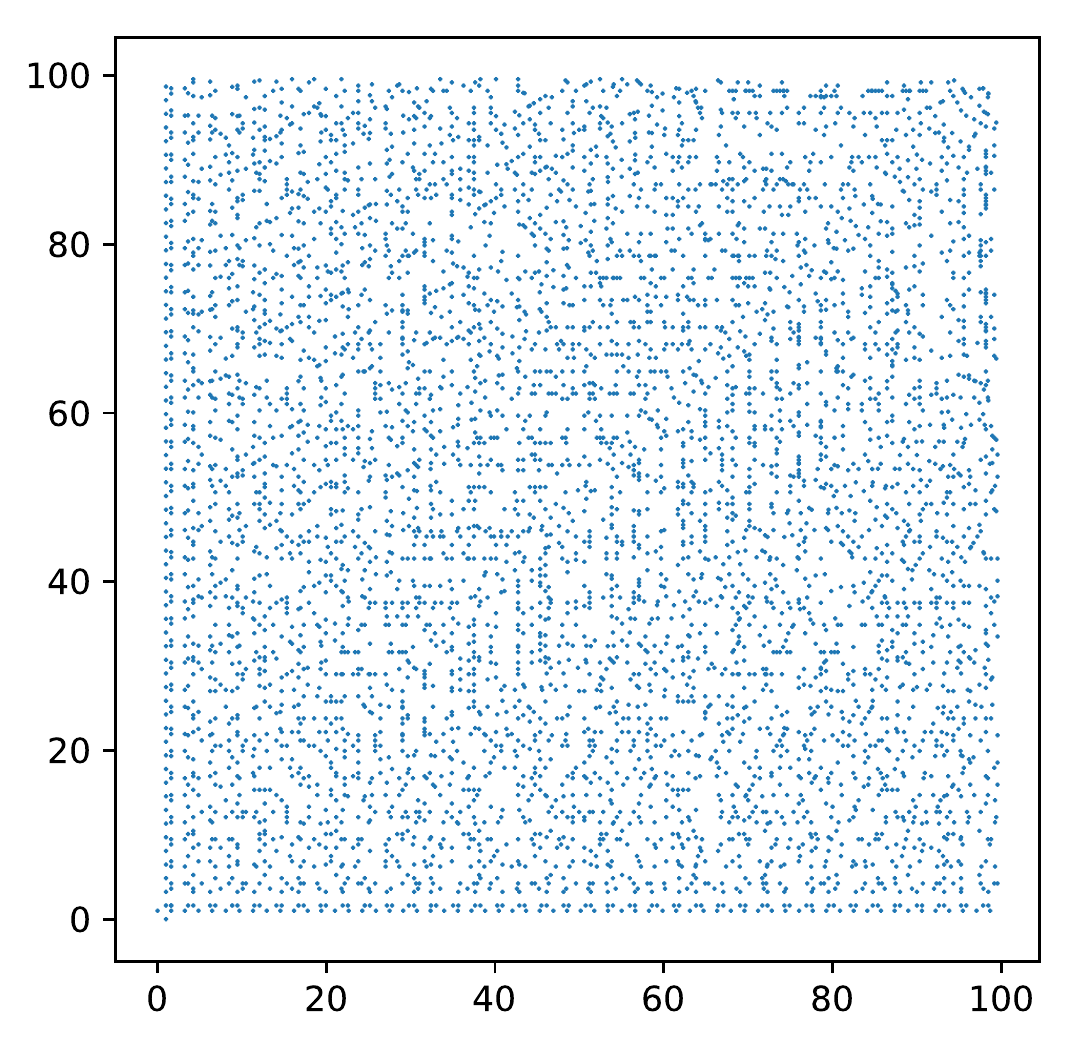}
\caption{The elements of $\Lambda_5$ in the square $[0, 100]^2$ generated using \cref{theorem: next term algorithm} and \cref{remark: practical next term algorithm}.}
\label{figure: G_5 vectors}
\end{figure}

An earlier version of this paper was announced in October 2018 under the title ``The Golden L Ford Circles'', which only considered $G_5$ and its Ford circles. An excellent paper \cite{Davis2018-al} by D. Davis and S. Lelievre that investigates the $G_5$-Stern-Brocot tree as a tool for studying the periodic paths on the pentagon, double pentagon, and golden L surfaces was announced at the same time. We strongly recommend the aforementioned paper as a more geometrically flavored application of the said trees.

\subsection{Organization}

This paper is organized as follows:

\begin{itemize}
\item In \cref{section: discrete orbits}, we characterize and study the discrete orbits of the linear action of $G_q$ on the plane $\RR^2$ (\cref{proposition: column - unimodular pair identification}), show that those discrete orbits have a tree structure analogous to the Stern-Brocot trees for the rationals (\cref{theorem: G_q Stern Brocot process is well-defined and exhaustive}), and derive the Boca-Cobeli-Zaharescu map analogues for $G_q$ (\cref{theorem: G_q BCZ maps}). We also characterize the periodic points for the $G_q$-BCZ map analogues (\cref{corollary: characterizing BCZ periodic points}), and present an algorithm for generating the elements of $\Lambda_q$ in increasing order of slope (\cref{theorem: next term algorithm}). We also collect some consequences of the existence of $G_q$-Stern-Brocot trees that we use throughout the paper in \cref{corollary: odds and ends}.
\item In \cref{section: cross section}, we give the Poincar\'{e} cross sections to the horocycle flow on the quotients $\SL(2, \RR)/G_q$ corresponding to the $G_q$ BCZ map analogues we have in section 2 (\cref{theorem: cross section}). As a consequence, we get an equidistribution result (\cref{theorem: weak limit and asymptotic growth}) that we use for the applications in \cref{section: applications}.
\item In \cref{section: applications}, we present a number of applications of the results in this paper to the statistics of subsets of $\Lambda_q$. In particular, we give the main asymptotic term for the number of elements of $\Lambda_q$ in homothetic dilations of triangles (\cref{proposition: counting in triangles}), equidistribution of homothetic dilations $\frac{1}{\tau}\Lambda_q$ in the square $[-1, 1]^2$ as $\tau \to \infty$ (\cref{corollary: equidistribution in the square}), the slope gap distribution for the elements of $\Lambda_q$ (\cref{corollary: limiting distribution of slopegap}), and the distribution of the Euclidean distance between successive $G_q$-Ford circles (\cref{corollary: limiting distribution of centdist}). We also get a weak form of the Dirichelet approximation theorem for $\Lambda_q$ for free (\cref{proposition: weak Dirichelet approximation}).
\end{itemize}

\subsection{Notation}

As is customary when working with the groups $G_q$, we write
\begin{equation*}
U_q := T_q S = \begin{pmatrix}\lambda_q & -1 \\ 1 & 0\end{pmatrix}.
\end{equation*}
The matrix $U_q$ is conjugate to a rotation with angle $\pi/q$, and can be easily seen to preserve the quadratic form
\begin{equation*}
  Q_q((x,y)^T) = x^2 - \lambda_q xy + y^2
\end{equation*}
when $U_q$ acts linearly on the plane $\RR^2$.

The main object that we study in this paper is the orbit of the vector $(1, 0)^T \in \RR^2$ under the linear action of $G_q$ on the plane
\begin{equation*}
\Lambda_q = G_q(1, 0)^T.
\end{equation*}
The set $\Lambda_q$ is symmetric against the lines $y = \pm x$, $x = 0$, and $y = 0$ since $G_q$ contains $S^3 T_q^{-1} S = T_q^T$, $S^3 = S^T$, and $(T_qS)^q = -\operatorname{Id}_2$.

Of special significance to us are the elements
\begin{equation*}
\mathfrak{w}_i^q = (x_i^q, y_i^q) = U_q^i (1, 0)^T,
\end{equation*}
where $i = 0 , 1, \cdots, 2 q - 1$. Note that $\mathfrak{w}_0^q = (1, 0)^T$, $\mathfrak{w}_1^q = (\lambda_q, 1)^T$, $\mathfrak{w}_{q-2}^q = (1, \lambda_q)^T$, $\mathfrak{w}_{q-1}^q = (0, 1)^T$, and $\mathfrak{w}_q^q = (-1, 0)^T$. (Since $U_q$ is conjugate to a $\pi/q$-rotation, $U_q^q = -\operatorname{Id}_2$. This gives the last equality.) Moreover, the vectors $\{\mathfrak{w}_i^q\}_{i = 0}^{2q - 1}$ lie on the ellipse $Q_q((x, y)^T) = x^2 - \lambda_q x y + y^2 = 1$.

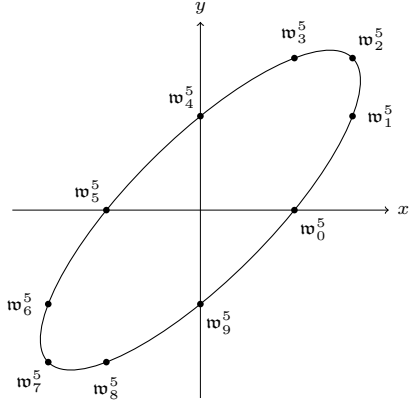
\begin{figure}
\centering
\begin{tikzpicture}[scale=1.25]
\draw[->,ultra thin] (-2,0)--(2,0) node[right]{$\scriptstyle x$};
\draw[->,ultra thin] (0,-2)--(0,2) node[above]{$\scriptstyle y$};
\draw[rotate=45] (0,0) ellipse (2.289cm and 0.743cm);
\fill (1,0) circle[radius=1pt];
\fill (1.618,1) circle[radius=1pt];
\fill (1.618,1.618) circle[radius=1pt];
\fill (1,1.618) circle[radius=1pt];
\fill (0,1) circle[radius=1pt];
\fill (-1,0) circle[radius=1pt];
\fill (-1.618,-1) circle[radius=1pt];
\fill (-1.618,-1.618) circle[radius=1pt];
\fill (-1,-1.618) circle[radius=1pt];
\fill (0,-1) circle[radius=1pt];
\node at (1+0.2,0-0.2) {$\scriptstyle \mathfrak{w}_0^5$};
\node at (1.618+0.3,1) {$\scriptstyle \mathfrak{w}_1^5$};
\node at (1.618+0.2,1.618+0.2) {$\scriptstyle \mathfrak{w}_2^5$};
\node at (1,1.618+0.275) {$\scriptstyle \mathfrak{w}_3^5$};
\node at (0-0.2,1+0.2) {$\scriptstyle \mathfrak{w}_4^5$};
\node at (-1-0.2,0+0.2) {$\scriptstyle \mathfrak{w}_5^5$};
\node at (-1.618-0.3,-1) {$\scriptstyle \mathfrak{w}_6^5$};
\node at (-1.618-0.2,-1.618-0.2) {$\scriptstyle \mathfrak{w}_7^5$};
\node at (-1,-1.618-0.275) {$\scriptstyle \mathfrak{w}_8^5$};
\node at (0+0.2,-1-0.2) {$\scriptstyle \mathfrak{w}_9^5$};
\end{tikzpicture}

\caption{The vectors $\{\mathfrak{w}_i^5\}_{i=0}^9$ along with the ellipse $Q_5((x, y)^T) = x^2 - \lambda_5 x y + y^2 = 1$. Note that $\lambda_5$ is the golden ratio $\varphi = (1+\sqrt{5})/2$.}
\end{figure}

Given two vectors $\mathbf{u}_0 = (x_0, y_0)^T, \mathbf{u}_1 = (x_1, y_1)^T \in \RR^2$, we denote their \emph{(scalar) wedge product} by
\begin{equation*}
  \mathbf{u}_0 \wedge \mathbf{u}_1 = x_0 y_1 - x_1 y_0,
\end{equation*}
and their \emph{dot product} by
\begin{equation*}
  \mathbf{u}_0 \cdot \mathbf{u}_1 = x_0 x_1 + y_0 y_1.
\end{equation*}
One useful inequality that we use more than once in this paper is that if $\mathbf{u}_0, \mathbf{u}_1, \mathbf{v}$ are non-zero vectors in $\RR^2$, with the angle $\angle \mathbf{u}_0 \mathbf{u}_1$ not exceeding $\pi/2$, and $\mathbf{v}$ belonging to the sector $(0, \infty)\mathbf{u}_0 + (0, \infty)\mathbf{u}_1 = \{\alpha \mathbf{u}_0 + \beta \mathbf{u}_1 \mid \alpha, \beta > 0\}$, then
\begin{equation}
\label{equation: sin inequalities}
0 < \frac{\mathbf{u}_0 \wedge \mathbf{v}}{\|\mathbf{u}_0\| \|\mathbf{v}\| }, \frac{\mathbf{v} \wedge \mathbf{u}_1}{\|\mathbf{v}\| \|\mathbf{u}_1\|} < \frac{\mathbf{u}_0 \wedge \mathbf{u}_1}{\|\mathbf{u}_0\| \|\mathbf{u}_1\|}.
\end{equation}
This follows from the identities $\mathbf{u}_0 \wedge \mathbf{v} = \|u_0\| \|\mathbf{v}\| \sin(\angle \mathbf{u}_0 \mathbf{v})$, $\mathbf{v} \wedge \mathbf{u}_1 = \|\mathbf{v}\| \|\mathbf{u}_1\| \sin(\angle\mathbf{v}\mathbf{u}_1)$, and $\mathbf{u}_0 \wedge \mathbf{u}_1 = \|\mathbf{u}_0\| \|\mathbf{u}_1\| \sin(\angle\mathbf{u}_0\mathbf{u}_1)$, in addition to the inequalities $\sin(\angle\mathbf{u}_0\mathbf{v}), \sin(\angle\mathbf{v}\mathbf{u}_1) < \sin(\angle\mathbf{u}_0\mathbf{u}_1)$. Finally, we say that the two vectors $\mathbf{u}_0, \mathbf{u}_1 \in \RR^2$ are \emph{unimodular} if $\mathbf{u}_0 \wedge \mathbf{u}_1 = 1$. For readability, we sometimes will denote the usual product on $\RR$ by $\times$. So $2 \times 3 = 6$, and so on.

Finally, we write
\begin{equation*}
h_s := \begin{pmatrix}1 & 0 \\ -s & 1\end{pmatrix},
\end{equation*}
for $s \in \RR$,
\begin{equation*}
s_\tau = \begin{pmatrix}\tau & 0 \\ 0 & \tau^{-1}\end{pmatrix},
\end{equation*}
for $\tau > 0$, and
\begin{equation*}
g_{a, b} = \begin{pmatrix}a & b \\ 0 & a^{-1}\end{pmatrix},
\end{equation*}
for $a > 0$, and $b \in \RR$. The above matrices satisfy the identities $g_{\tau, 0} = s_\tau$, $h_s h_t = h_{s + t}$, and $h_s s_\tau = s_\tau h_{s\tau^2}$.

\section{The Discrete Orbits, Stern-Brocot Trees, and Boca-Cobeli-Zaharescu Map Analogue for $\Lambda_q = G_q(1, 0)^T$}
\label{section: discrete orbits}

\subsection{The Discrete Orbits of the Linear Action of $G_q$ on the Plane $\RR^2$}

\begin{proposition}
\label{proposition: column - unimodular pair identification}
The following are true.
\begin{enumerate}
\item If the orbit of $\mathbf{u} \in \RR^2$ under the linear action of $G_q$ is a discrete subset of $\RR^2$, then either $\mathbf{u} = (0, 0)^T$, or $G_q \mathbf{u}$ is a homothetic dilation of $\Lambda_q =  G_q(1,0)^T$.

\item The ellipse $Q_q((x, y)^T) = x^2 - \lambda_q x y + y^2 = 1$ does not contain any elements of $\Lambda_q$ in its interior.

\item The elements $\mathfrak{w}_0^q, \mathfrak{w}_1^q, \cdots, \mathfrak{w}_{q-1}^q$ of $\Lambda_q$ satisfy the \emph{Farey neighbor identities}
\begin{equation*}
\mathfrak{w}_i^q \wedge \mathfrak{w}_{i+1}^q = 1,
\end{equation*}
for $i = 0, 1, \cdots, q - 2$, in addition to
\begin{equation*}
\mathfrak{w}_0^q \wedge \mathfrak{w}_{q-1}^q = 1.
\end{equation*}

\item If $\mathbf{u}_0, \mathbf{u}_1 \in \Lambda_q$ are two unimodular vectors (i.e $\mathbf{u}_0 \wedge \mathbf{u}_1 = 1$), then there exists $A \in G_q$ such that $A\mathbf{u}_0 = \mathfrak{w}_0^q = (1, 0)^T$ and $A\mathbf{u}_1 = \mathfrak{w}_{q-1}^q = (0, 1)^T$. That is, the pairs of unimodular vectors of $\Lambda_q$ are in a one-to-one correspondence with the columns of the matrices in $G_q$.

\end{enumerate}
\end{proposition}

\begin{proof}
For the first claim: Assume without loss of generality that $\mathbf{u} \neq (0, 0)^T$. Let $\Sigma_i^q = (0,\infty) \mathfrak{w}_i^q + [0,\infty) \mathfrak{w}_{i+1 \mod 2q}^q = \{\alpha \mathfrak{w}_i^q + \beta \mathfrak{w}_{i+1 \mod 2q}^q \mid \alpha, \beta > 0\}$, $i = 0, 1, \cdots, 2q - 1$, be the radial sectors of $\RR^2 \setminus \{(0, 0)^T\}$ defined by the directions $\{\mathfrak{w}_i^q\}_{i=0}^{2q-1}$. Note that the matrix $U_q$ bijectively maps each sector $\Sigma_i^q$ to the sector $\Sigma_{i + 1 \mod 2q}^q$ for $i = 0, 1, \cdots, 2q-1$, and maintains the values of the quadratic form $Q_q$ at each point. Also, $T_q^{-1}$ maps the sector $\Sigma_0^q$ to $\cup_{i=0}^{q-2} \Sigma_i = [0,\infty)(1,0)^T + (0, \infty)(1,0)^T$, decreasing the $Q_q$-values of all the points in the interior of $\Sigma_0$, and fixing all the points on the ray in the direction of $\mathfrak{w}_0^q = (1,0)^T$. (This follows from $T_q^{-1} \mathfrak{w}_0^q = \mathfrak{w}_0^q = (1, 0)^T$, and $T_q^{-1}\mathfrak{w}_1^q = \mathfrak{w}_{q-1}^q = (0,1)^T$.) Starting with the vector $\mathbf{u}$ whose $G_q$-orbit is being considered, we repeatedly apply the following process:
\begin{enumerate}
\item If $\mathbf{u} \in \Sigma_i^q$ for some $1 \leq i \leq 2q - 1$, then replace $\mathbf{u}$ with $U_q^{-i} \mathbf{u} \in \Sigma_0^q$. This maintains the $Q_q$-value of $\mathbf{u}$.

\item Replace $\mathbf{u} \in \Sigma_0^q$ with $T_q^{-1}\mathbf{u} \in \cup_{i=0}^{q-2}\Sigma_i^q$. This fixes $\mathbf{u}$ if it lies on the ray in the direction of $\mathfrak{w}_0^q = (1, 0)^T$, and otherwise reduces the $Q_q$-value of $\mathbf{u}$.
\end{enumerate}
After each iteration of this process, either the point $\mathbf{u}$ lands on the line $y = 0$ and is fixed by further applications of the process, or is mapped to another point in $G_q\mathbf{u}$ with a strictly smaller $Q_q$-value. By the discreteness of $G_q\mathbf{u}$, the point $\mathbf{u}$ will eventually land on the line $y = 0$. This implies that there exists a non-zero $\alpha \in \RR$ such that $\mathbf{u} \in \alpha G_q (1, 0)^T$, from which follows that $G_q\mathbf{u} = \alpha \Lambda_q$. This proves the first claim.

The second claim follows from the fact that $(1, 0)^T \in \Lambda_q$ lies on the ellipse $Q_q((x, y)^T) = x^2 - \lambda_q x y + y^2 = 1$. No point in $\Lambda_q$ can have a $Q_q$-value smaller than $1$, as the iterative process used above will produce an element of $\Lambda_q$ that is parallel to $(1, 0)$ and shorter than it, which cannot happen by the discreteness of $\Lambda_q$.

For the third claim: we have for all $i = 0, 1, \cdots, q - 2$ that
\begin{equation*}
\mathfrak{w}_i^q \wedge \mathfrak{w}_{i+1}^q = (U_q^i \mathfrak{w}_0^q) \wedge (U_q^i \mathfrak{w}_1^q) = \det(U_q^i) \times \mathfrak{w}_0^q \wedge \mathfrak{w}_1^q = 1 \times (1,0)^T \wedge (\lambda_q, 1)^T = 1.
\end{equation*}
We also have that
\begin{equation*}
\mathfrak{w}_0^q \wedge \mathfrak{w}_{q-1}^q = (1,0)^T \wedge (0, 1)^T = 1.
\end{equation*}

For the fourth claim: By definition, there exists $B \in G_q$ such that $B\mathbf{u}_0 = (1, 0)^T$. Acting by $B^{-1}$, the two vectors $\widetilde{\mathbf{u}}_0 = B^{-1}\mathbf{u}_0 = (1, 0)^T$ and $\widetilde{\mathbf{u}}_1 = B^{-1}\mathbf{u}_1$ satisfy
\begin{equation*}
\widetilde{\mathbf{u}}_0 \wedge \widetilde{\mathbf{u}}_1 = \det(B^{-1}) \times \mathbf{u}_0 \wedge \mathbf{u}_1 = 1.
\end{equation*}
If $\widetilde{\mathbf{u}}_1 = (x, y)$, then $y = 1$. Shearing by $T_q$, we have $T_q^n\widetilde{\mathbf{u}}_0 = \widetilde{\mathbf{u}}_0$, and $T_q^n \widetilde{\mathbf{u}}_1 = (x + n \lambda_q, 1)^T$ for all $n \in \ZZ$. Since $\mathfrak{w}_1^q = (\lambda_q, 1)^T$ and $\mathfrak{w}_{q-1}^q = (0,1)^T$ are two elements of $\Lambda_q$ on the ellipse $Q_q = 1$, are at height $y = 1$, are a horizontal distance $\lambda_q$ away from each other, and $T_q^{-1}\mathfrak{w}_0^q = \mathfrak{w}_{q-1}^q$, then there exists $n_0 \in \ZZ$ such that $T_q^{n_0} \widetilde{\mathbf{u}}_1 = \mathfrak{w}_{q-1}^q$. Now, taking $A = T_q^{n_0}B^{-1}$ proves the claim.
\end{proof}

\subsection{The Stern-Brocot Trees for $\Lambda_q = G_q(1,0)^T$}

\begin{definition}
We refer to the process of iteratively replacing a pair of vectors $\mathbf{u}_0, \mathbf{u}_1 \in \Lambda_q$ that are unimodular (i.e. $\mathbf{u}_0 \wedge \mathbf{u}_1 = 1$) with the vectors
\begin{equation*}
x_0^q \mathbf{u}_0 + y_0^q \mathbf{u}_1 = \mathbf{u}_0, x_1^q \mathbf{u}_0 + y_1^q \mathbf{u}_1, \cdots, x_{q-2}^q \mathbf{u}_0 + y_{q-2}^q \mathbf{u}_1, x_{q-1}^q \mathbf{u}_0 + y_{q-1}^q \mathbf{u}_1 = \mathbf{u}_1
\end{equation*}
as the \emph{$G_q$-Stern-Brocot process}. We refer to the vectors $\{x_i^q \mathbf{u}_0 + y_i^q \mathbf{u}_1\}_{i=1}^{q-2}$ as the \emph{($G_q$-Stern-Brocot) children} of $\mathbf{u}_0, \mathbf{u}_1$, and successive children of the children of $\mathbf{u}_0, \mathbf{u}_1$ as the \emph{($G_q$-Stern-Brocot) grandchildren} of $\mathbf{u}_0, \mathbf{u}_1$.
\end{definition}

\begin{theorem}
\label{theorem: G_q Stern Brocot process is well-defined and exhaustive}
Let $\mathbf{u}_0, \mathbf{u}_1 \in \Lambda_q$ be two unimodular vectors (i.e. $\mathbf{u}_0 \wedge \mathbf{u}_1 = 1$). The $G_q$-Stern-Brocot process applied to $\mathbf{u}_0$ and $\mathbf{u}_1$ generates a well-defined tree of elements of $\Lambda_q$, and exhausts the elements of $\Lambda_q$ in the sector $[0,\infty)\mathbf{u}_0 + [0,\infty)\mathbf{u}_1 = \{\alpha \mathbf{u}_0 + \beta \mathbf{u}_1 \mid \alpha, \beta \geq 0\}$.
\end{theorem}

\begin{proof}
That the Stern-Brocot process is well-defined for any two unimodular elements $\mathbf{u}_0$ and $\mathbf{u}_1$ of $\Lambda_q$ follows from \cref{proposition: column - unimodular pair identification}. In particular, since $\mathbf{u}_0$ and $\mathbf{u}_1$ are unimodular, then there exists $A \in G_q$ whose columns are $\mathbf{u}_0$ and $\mathbf{u}_1$ (i.e. $A(1,0)^T = \mathbf{u}_0$ and $A(0,1)^T = \mathbf{u}_1$). The vectors $\mathfrak{w}_i^q = (x_i^q, y_i^q)^T = x_i^q (1, 0)^T + y_i^q (0, 1)^T$, with $i = 0, 1, \cdots, q - 1$, are unimodular in pairs (by the Farey neighbor identities from \cref{proposition: column - unimodular pair identification}), and so their images $A\mathfrak{w}_i^q = x_i^q \mathbf{u}_0 + y_i^q \mathbf{u}_1$, $i = 0, 1, \cdots, q - 1$, satisfy the same Farey neighbor identities, are all elements of $\Lambda_q$, and all belong to the sector $[0,\infty)\mathbf{u}_0 + [0,\infty)\mathbf{u}_1$. It remains to prove that the Stern-Brocot process is exhaustive, and our proof is similar to that of the classical proof for Farey fractions.

We first need to show that the wedge products of pairs of non-parallel elements of $\Lambda_q$ are bounded away from zero.\footnote{This can be trivially extended into a proof that the set of wedge products of the elements of $\Lambda_q$ is discrete, similar to a characterization of lattice surfaces from \cite{Smillie2010-nb}.} Given two elements $\mathbf{w}_0, \mathbf{w}_1$ of $\Lambda_q$, we assume that if $0 < \mathbf{w}_0 \wedge \mathbf{w}_1 < \epsilon$, then $\epsilon$ cannot be arbitrarily small. Pick any $A \in G_q$ with $A\mathbf{u}_0 = (1, 0)^T$. Writing $A\mathbf{u}_1 = (x, y)^T$, then $0 < A\mathbf{u}_0 \wedge A\mathbf{u}_1 = y < \epsilon$. Shearing by $T_q^{\pm} = \begin{pmatrix}1 & \pm\lambda_q \\ 0 & 1\end{pmatrix}$, we can find $n \in \mathbb{Z}$ such that $T_q^n A\mathbf{u}_1 = (x + n\lambda_qy, y)^T$ has an $x$-component $0 \leq x + n \lambda_q y < \lambda_q \epsilon$. From this follows that $\|T_q^n A \mathbf{u}_1\| \leq \epsilon\sqrt{1 + \lambda_q^2}$, and so $\epsilon$ cannot be arbitrarily small by the discreteness of $\Lambda_q$. It thus follows that for all $q \geq 3$, there exists $\epsilon_q$ such that the wedge product of any non-parallel pair of elements of $\Lambda_q$ is bounded below by $\epsilon_q$ in absolute value.

Now, we write $\mathbf{u}_0 = (q_0,a_0)^T$, and $\mathbf{u}_1 = (q_1,a_1)^T$, and assume that $\mathbf{u}_0, \mathbf{u}_1$ belong to the first quadrant. (We can safely do that by the last claim of \cref{proposition: column - unimodular pair identification}.) If $(x,y)^T \in \Lambda_q$ belongs to the sector $(0,\infty)\mathbf{u}_0 + (0,\infty)\mathbf{u}_1$, the orientation of the vectors gives $\mathbf{u}_0 \wedge (x,y)^T, (x,y)^T \wedge \mathbf{u}_1 > 0$, and so $\mathbf{u}_0 \wedge (x,y)^T, (x,y)^T \wedge \mathbf{u}_1 \geq \epsilon_q$. We define the component sum function $\varsigma : \RR^2 \to \RR$ by $\varsigma(r,s)^T = r + s$ for all $(r,s)^T \in \RR^2$. We thus get
\begin{eqnarray*}
\varsigma(\mathbf{u}_1) \left(\mathbf{u}_0 \wedge (x,y)^T\right) + \varsigma(\mathbf{u}_0) \left((x,y)^T \wedge \mathbf{u}_1\right) &=& \begin{aligned} & (a_1 + q_1)(y q_0 - x a_0) \\ & \ + (a_0 + q_0)(a_1 x - q_1 y) \end{aligned} \\
 &=& (a_1 q_0 - a_0 q_1) (x + y) \\
 &=& \mathbf{u}_0 \wedge \mathbf{u}_1 \times \varsigma(x,y)^T \\
 &=& \varsigma(x,y)^T,
\end{eqnarray*}
and so
\begin{equation}
\label{equation: varsigma bound}
\varsigma(x,y)^T \geq \epsilon_q \left(\varsigma(\mathbf{u}_0) + \varsigma(\mathbf{u}_1)\right).
\end{equation}
Assuming without loss of generality that we are starting the Stern-Brocot process with $(1,0)^T$ and $(0,1)^T$, we have that the $\varsigma$ value of any vector that is generated at the $n$th step, $n \geq 0$, is bounded below by $n + 1$. (We demonstrate this fact at the end of this proof.) At any step, if $(x,y)^T$ is not one of the $q-2$ Stern-Brocot children of $\mathbf{u}_0$ and $\mathbf{u}_1$, then it belongs to a sector defined by one of the $q-1$ pairs of successive unimodular vectors that have been generated at this step. This cannot take place forever as each step of Stern-Brocot increases the right hand side of \cref{equation: varsigma bound} by at least $\epsilon_q$. This implies that $(x,y)^T$ eventually shows up as a child, and we are done.

Now we prove the lower bound on the $\varsigma$ value. If $\mathbf{c}$ is the $G_q$-Stern-Brocot child of two vectors $\mathbf{p}_1, \mathbf{p}_2$ in the first quadrant, then $\mathbf{c} = x_{i_0}^q \mathbf{p}_1 + y_{i_0}^q \mathbf{p}_2$ for some $1 \leq i_0 \leq q - 2$, and so $\varsigma(\mathbf{c}) = x_{i_0}^q \varsigma(\mathbf{p}_1) + y_{i_0}^q \varsigma(\mathbf{p}_2) \geq \varsigma(\mathbf{p}_1) + \varsigma(\mathbf{p}_2)$, since $x_{i_0}^q, y_{i_0}^q \geq 1$. It is easy to see that each of the vectors that are generated at one stage must have at least one parent that was generated at the previous stage. Since $\varsigma((1, 0)^T), \varsigma((0, 1)^T) = 1$, it now follows by induction that the $\varsigma \geq n + 1$ for all the vectors that are generated at the $n$th stage for $n \geq 0$.
\end{proof}

In the following corollary, we collect some consequences of the existence of Stern-Brocot tree for $\Lambda_q$ that we use in the remainder of this paper.

\begin{corollary}
\label{corollary: odds and ends}
The following are true.
\begin{enumerate}
\item If $\mathbf{v}_0, \mathbf{v}_1 \in \Lambda_q$ are such that $\mathbf{v}_0 \neq \pm \mathbf{v}_1$, then $|\mathbf{v}_0 \wedge \mathbf{v}_1| \geq 1$.

\item Let $\mathbf{v} \in \RR^2 \setminus \{(0, 0)^T\}$ be an arbitrary non-zero vector in the plane. Then either $\mathbf{v}$ is parallel to a vector in $\Lambda_q$, or for any unimodular pair $\mathbf{u}_0, \mathbf{u}_1 \in \Lambda_q$, if $\mathbf{v}$ belongs to the sector $(0,\infty)\mathbf{u}_0 + (0, \infty)\mathbf{u}_1$, then there exists a pair of unimodular $G_q$-Stern-Brocot grandchildren $\mathbf{w}_0, \mathbf{w}_1$ of $\mathbf{u}_0, \mathbf{u}_1$ such that $\mathbf{v}$ belongs to the sector $(0,\infty)\mathbf{w}_0 + (0, \infty)\mathbf{w}_1$, and $\mathbf{w}_0, \mathbf{w}_1$ are different from $\mathbf{u}_0, \mathbf{u}_1$.

\item Let $\mathbf{u}_0, \mathbf{u}_1 \in \Lambda_q$ be two unimodular vectors, and $\{\mathbf{w}_n\}_{n=1}^\infty$ be any sequence of elements of $\Lambda_q$ such that for each $n \geq 1$, $\mathbf{w}_n$ is generated at the $n$th iteration of the $G_q$-Stern-Brocot process applied to the two unimodular vectors $\mathbf{u}_0, \mathbf{u}_1$. Then $\lim_{n \to \infty} \|\mathbf{w}_n\| = \infty$.

\item The slopes of the non-vertical vectors in $\Lambda_q$ are dense in $\RR$.
\end{enumerate}
\end{corollary}

\begin{proof}
We first prove the following: If $\mathbf{u}_0, \mathbf{u}_1 \in \Lambda_q$ are unimodular (i.e. $\mathbf{u}_0 \wedge \mathbf{u}_1 = 1$), then after $n \geq 1$ applications of the $G_q$-Stern-Brocot process, the two vectors $\mathbf{w}_n^r = n \lambda_q \mathbf{u}_0 + \mathbf{u}_1 = (n\lambda_q, 1)^T$ and $\mathbf{w}_n^l = \mathbf{u}_0 + n \lambda_q \mathbf{u}_1 = (1, n\lambda_q)^T$ are $G_q$-Stern-Brocot grandchildren of $\mathbf{u}_0$ and $\mathbf{u}_1$, and all the grandchildren of $\mathbf{u}_0$ and $\mathbf{u}_1$ that have been generated by the $n$th step belong to the sector $[0,\infty)\mathbf{w}_n^r + [0,\infty)\mathbf{w}_n^l$. Now, since $(x_1^q, y_1^q)^T = \mathfrak{w}_1^q = U_q(1, 0)^T = (\lambda_q, 1)^T$, and $(x_{q-2}^q, y_{q-2}^q)^T = \mathfrak{w}_{q-2}^q = U_q^{-1}(0,1)^T = (1, \lambda_q)^T$, it follows from \cref{theorem: G_q Stern Brocot process is well-defined and exhaustive} that the two vectors $\mathbf{w}_1^r = x_1^q \mathbf{u}_0 + y_1^q \mathbf{u}_1 = \lambda_q \mathbf{u}_0 + \mathbf{u}_1$ and $\mathbf{w}_1^l = x_{q-2}^q \mathbf{u}_0 + y_{q-2}^q \mathbf{u}_1 = \mathbf{u}_0 + \lambda_q \mathbf{u}_1$ are Stern-Brocot children of $\mathbf{u}_0$ and $\mathbf{u}_1$, and that all the children of $\mathbf{u}_0$ and $\mathbf{u}_1$ that were generated after one iteration are contained in the sector corresponding to $\mathbf{w}_1^r$ and $\mathbf{w}_1^l$. The remainder of the claim follows by repeatedly applying the Stern-Brocot process to the unimodular pair $\mathbf{u}_0$ and $\mathbf{w}_n^r$, and the unimodular pair $\mathbf{w}_n^l$ and $\mathbf{u}_1$, for all $n \geq 2$.

For the first claim: Since $-\operatorname{Id}_2 = U_q^q$ is in $G_q$, we can assume that the angle between $\mathbf{v}_0$ and $\mathbf{v}_1$ does not exceed $\pi/2$. We also permute $\mathbf{v}_0$ and $\mathbf{v}_1$ if need be so that $\mathbf{v}_0 \wedge \mathbf{v}_1 > 0$. Furthermore, we can assume that $\mathbf{v}_0 = (1, 0)^T$. (There exists $A \in G_q$ such that $\mathbf{v}_0 = A(1, 0)^T$, and so we can replace $\mathbf{v}_0$ and $\mathbf{v}_1$ with $\widetilde{\mathbf{v}}_0 = A^{-1}\mathbf{v}_0$ and $\widetilde{\mathbf{v}}_1 = A^{-1}\mathbf{v}_1$, and preserve the wedge product $\widetilde{\mathbf{v}}_0 \wedge \widetilde{\mathbf{v}}_1 = \det(A^{-1}) \times \mathbf{v}_0 \wedge \mathbf{v}_1 = \mathbf{v}_0 \wedge \mathbf{v}_1$.) We now have that $\mathbf{v}_0 = (1, 0)^T$, and that $\mathbf{v}_1$ is in the first quadrant. That is, $\mathbf{v}_1$ is either $\mathfrak{w}_{q-1}^q = (0, 1)^T$, or a $G_q$-Stern-Brocot grandchild of $\mathfrak{w}_0^q = (1, 0)^T$ and $\mathfrak{w}_{q-1}^q = (0, 1)^T$. Writing $\mathbf{v}_1 = (x_{\mathbf{v}_1}, y_{\mathbf{v}_1})^T$, we have $\mathbf{v}_0 \wedge \mathbf{v}_1 = y_{\mathbf{v}_1}$. The $y$ components of the vectors $\{\mathbf{w}_n^r\}_{n=1}^\infty$ from the first claim in the corollary all are all $y = 1$, and so $y_{\mathbf{v}_1} = \mathbf{v}_0 \wedge \mathbf{v}_1 \geq 1$ as required.

For the second claim: The unit vectors in the directions of $\{\mathbf{w}_n^r\}_{n=1}^\infty$ and $\{\mathbf{w}_n^l\}_{n=1}^\infty$ converge to $\mathbf{u}_0$ and $\mathbf{u}_1$ as $n \to \infty$. As such, if the vector $\mathbf{v}$ is not in $\Lambda_q$, then it will eventually be contained in the sector bounded by $\mathbf{w}_{n_0}^l$ and $\mathbf{w}_{n_0}^l$ for some $n_0 \geq 1$, and consequently belongs to the sector bounded by a pair of unimodular grandchildren of $\mathbf{u}_0$ and $\mathbf{u}_1$.

For the third claim: By the fourth claim in \cref{proposition: column - unimodular pair identification}, and the boundedness of the elements of $G_q$ as linear operators on $\RR^2$, we can assume without loss of generality that $\mathbf{u}_0 = (1, 0)^T$ and $\mathbf{u}_1 = (0, 1)^T$. At the end of the proof of \cref{theorem: G_q Stern Brocot process is well-defined and exhaustive}, we showed that if $\mathbf{w}_n$ is generated at the $n$th stage of the $G_q$-Stern-Brocot process applies to $(1, 0)^T$ and $(0, 1)^T$, then $\mathbf{w}_n \geq n + 1$. If $\mathbf{w}_n = (r, s)^T$, then $\varsigma(\mathbf{w}_n) = r + s \leq \sqrt{2} \sqrt{r^2 + s^2} \leq \sqrt{2} \|\mathbf{w}_n\|$, which proves the claim.

For the fourth claim: It suffices to show that if $\alpha \geq 0$ is not the slope of a vector in $\Lambda_q$, then $\alpha$ can be approximated by slopes of vectors in $\Lambda_q$. Writing $\mathbf{v} = (1, \alpha)^T$, we note that if $\mathbf{u}_0, \mathbf{u}_1 \in \Lambda_q$ are two unimodular vectors in the first quadrant whose sector contains $\mathbf{v}$, then by \cref{equation: sin inequalities} we have
\begin{equation*}
0 < \mathbf{u}_0 \wedge \mathbf{v} < \frac{\|\mathbf{v}\|}{\|\mathbf{u}_1\|} \mathbf{u}_0 \wedge \mathbf{u}_1 = \mathbf{v} < \frac{\|\mathbf{v}\|}{\|\mathbf{u}_1\|}.
\end{equation*}
Writing $\mathbf{u}_0 = (x, y)^T$, and assuming that $x > 0$, we thus get
\begin{equation}
\label{equation: Diophantine approximation}
0 \leq \alpha - \frac{y}{x} \leq \frac{\sqrt{1+\alpha^2}}{x\|\mathbf{u}_1\|}.
\end{equation}
Now, we can start with $\mathbf{u}_0 = (1, 0)^T$ and $\mathbf{u}_1 = (0, 1)^T$ as two vectors in the first quadrant whose sector contains $\mathbf{v}$, and by the second claim in this corollary, we can repeatedly replace $\mathbf{u}_0$ and $\mathbf{u}_1$ with unimodular pairs that are generated at later stages of the Stern-Brocot process. In \cref{equation: Diophantine approximation}, $x \geq 1$, and $\lim_{n \to \infty} \|\mathbf{u}_1\| = \infty$, and we are done. 
\end{proof}

\subsection{The Boca-Cobeli-Zaharescu Map Analogue for $\Lambda_q = G_q(1,0)^T$}

In the following theorem, we present the BCZ map analogue for $\Lambda_q$. In essence, this theorem along with the next-term algorithm (\cref{theorem: next term algorithm}) extend the properties of the Farey sequence alluded to in the introduction using the BCZ map formalization.

\begin{theorem}
\label{theorem: G_q BCZ maps}
The following are true.
\begin{enumerate}
\item For any $A \in \SL(2, \RR)$, if $A \Lambda_q$ has a horizontal vector of length not exceeding $1$ (i.e. a horizontal vector in $A \Lambda_q \cap S_1$), then $A \Lambda_q$ can be uniquely identified with a point $(a_A, b_A)$ in the \emph{$G_q$-Farey triangle}
\begin{equation*}
\mathscr{T}^q = \{(a, b) \in \RR^2 \mid 0 < a \leq 1,\ 1 - \lambda_q a < b \leq 1\}
\end{equation*}
through $B\Lambda_q = g_{a_A,b_A}\Lambda_q$. Moreover, the value $a_A$ agrees with the length of the horizontal vector in $A\Lambda_q \cap S_1$.

\item Let $(a, b) \in \mathscr{T}^q$ be any point in the $G_q$-Farey triangle. The set $g_{a, b} \Lambda_q \cap S_1$ has a vector with smallest positive slope. Consequently, there exists a smallest $s = R_q(a, b) > 0$ such that $h_s g_{a,b} \Lambda_q$ has a horizontal vector of length not exceeding $1$, and hence $h_s g_{a, b} \Lambda_q$ corresponds to a unique point $\BCZ_q(a, b) \in \mathscr{T}^q$ in the $G_q$-Farey triangle. The function $R_q : \mathscr{T}^q \to \RR_+$ is referred to as the \emph{$G_q$-roof function}, and the map $\BCZ_q(a, b) : \mathscr{T}^q \to \mathscr{T}^q$ is referred to as the \emph{$G_q$-BCZ map}.

\item The $G_q$-Farey triangle $\mathscr{T}^q$ can be partitioned into the union of
\begin{equation*}
\mathscr{T}_i^q := \{(a, b) \in \mathscr{T}^q \mid (a, b)^T \cdot \mathfrak{w}_{i - 1} > 1,\ (a, b)^T\cdot\mathfrak{w}_i \leq 1\},
\end{equation*}
with $i = 2, 3, \cdots, q - 1$, such that if $(a, b) \in \mathscr{T}_i^q$, then $g_{a, b}\mathfrak{w}_i^q$ is the vector of least positive slope in $g_{a,b}\Lambda_q \cap S_1$, and
\begin{itemize}
\item the value of the roof function $R_q(a, b)$ is given by
\begin{equation*}
R_q(a, b) = \frac{y_i^q}{a \times (a,b)^T \cdot \mathfrak{w}_i^q}, \text{ and }
\end{equation*}
\item the value of the BCZ map $\BCZ_q(a, b)$ is given by
\begin{equation*}
\BCZ_q(a, b) = \left((a,b)^T \cdot \mathfrak{w}_i^q, (a,b)^T \cdot \mathfrak{w}_{i+1}^q + k_i^q(a, b) \times \lambda_q \times (a, b)^T \cdot \mathfrak{w}_i^q\right),
\end{equation*}
where the \emph{$G_q$-index} $k_i^q(a,b)$ is given by
\begin{equation*}
k_i^q(a,b) = \left\lfloor \frac{1 - (a,b)^T \cdot \mathfrak{w}_{i+1}^q}{\lambda_q \times (a,b)^T \cdot \mathfrak{w}_i^q} \right\rfloor.
\end{equation*}
\end{itemize}
\end{enumerate}
\end{theorem}

\begin{figure}
\centering
\begin{tikzpicture}[scale=4]
\draw[->,ultra thin] (-0.5,0)--(1.25,0) node[right]{$\scriptstyle a$};
\draw[->,ultra thin] (0,-0.5)--(0,1.25) node[above]{$\scriptstyle b$};
\draw[thick] (0, 1) -- (1, -0.618) -- (1, 1) -- (0, 1);
\draw[thick] (0.382, 0.382) -- (1, 0);
\draw[thick] (0.618, 0) -- (1, -0.382);
\node at (1.2, 0.5) {$\mathscr{T}_4^5$};
\node at (1.2, -0.191) {$\mathscr{T}_3^5$};
\node at (1.2, -0.5) {$\mathscr{T}_2^5$};
\node at (1.1, 0.191) [rotate=90] {$a = 1$};
\node at (0.5, 1.09) {$\mathcal{L}_4^{5, \text{top}}$};
\node [rotate=-31.27] at (0.69, 0.3) {$\mathcal{L}_3^{5, \text{top}}$};
\node [rotate=-45] at (0.81, -0.05) {$\mathcal{L}_2^{5, \text{top}}$};
\node [rotate=-58.28] at (0.4, 0.1) {$\mathcal{L}_1^{5, \text{top}}$};
\end{tikzpicture}

\caption{The $G_5$-Farey triangle $\mathscr{T}^5$ with the subregions $\mathscr{T}_2^5$, $\mathscr{T}_3^5$, and $\mathscr{T}_4^5$ from \cref{theorem: G_q BCZ maps} indicated. The figure also shows the lines $\mathcal{L}_i^{5, \text{top}} = \{(a, b) \in \mathscr{T}^5 \mid (a, b)^T \cdot \mathfrak{w}_i^q = 1\}$, $i = 1, 2, 3, 4$, that bound the aforementioned subregions in the proof of \cref{theorem: G_q BCZ maps}.}
\end{figure}
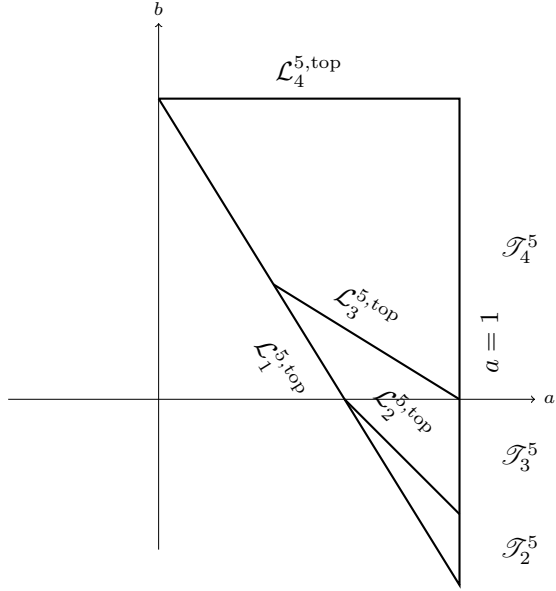

We first need the following lemma.

\begin{lemma}
\label{lemma: orbit-stabilizer lemma}
Given $A \in \SL(2, \RR)$, if $A\Lambda_q$ contains both $(1, 0)^T$ and $(0, 1)^T$, then $A\Lambda_q = \Lambda_q$. In particular, the following are true.

\begin{enumerate}
\item  For any $B \in \SL(2, \RR)$, $B \Lambda_q = \Lambda_q$ if and only if $B \in G_q$. From this follows that the sets $C\Lambda_q$, with $C$ varying over $\SL(2, \RR)$, can be identified with the elements of $\SL(2, \RR)/G_q$.

\item For any $B \in \SL(2, \RR)$, if $B \Lambda_q$ contains a horizontal vector of length $a > 0$, then there exists $b \in \RR$ such that $B \Lambda_q = g_{a,b} \Lambda_q$.
\end{enumerate}
\end{lemma}

\begin{proof}
We first prove the main claim. Let $A \in \SL(2, \RR)$ be such that $A \Lambda_q$ contains both $(1, 0)^T$ and $(0, 1)^T$. Then there exists $\mathbf{u}_0, \mathbf{u}_1 \in \Lambda_q$ such that $\mathbf{u}_0 = A^{-1}(1, 0)^T$ and $\mathbf{u}_1 = A^{-1}(0, 1)^T$, and $\mathbf{u}_0 \wedge \mathbf{u}_1 = \det(A^{-1}) \times (1, 0)^T \times (0, 1)^T = 1$. The columns of $A^{-1}$ thus form a unimodular pair of elements of $\Lambda_q$, and so by the last claim of \cref{proposition: column - unimodular pair identification}, the matrix $A^{-1}$, and by necessity $A$, belong to the group $G_q$.

The first claim now follows from the fact that if $B \in \SL(2, \RR)$ is such that $B\Lambda_q = \Lambda_q$, then $B\Lambda_q$ contains both $(1, 0)^T$ and $(0, 1)^T$.

We now prove the second claim. Let $\mathbf{u}_0 \in \Lambda_q$ be such that $B\mathbf{u}_0 = (a, 0)^T \in B\Lambda_q$ is parallel to the horizontal vector in question. If $A \in \SL(2, \RR)$ is such that $\mathbf{u}_0 = A(1, 0)^T$, then $\mathbf{u}_1 = A(0, 1)^T$ is an element of $\Lambda_q$ with $\mathbf{u}_0 \wedge \mathbf{u}_1 = 1$. Writing $\widetilde{\mathbf{u}}_0 = \begin{pmatrix}a^{-1} & 0 \\ 0 & a\end{pmatrix}B\mathbf{u}_0 = (1, 0)^T$, and $\widetilde{\mathbf{u}}_1 = \begin{pmatrix}a^{-1} & 0 \\ 0 & a\end{pmatrix}B\mathbf{u}_1 = (x, y)^T$, we have that $\widetilde{\mathbf{u}}_0 \wedge \widetilde{\mathbf{u}}_1 = 1$, and so $y = 1$. Shearing by $T_{-x} = \begin{pmatrix}1 & -x \\ 0 & 1\end{pmatrix}$, we have that $T_{-x} \widetilde{\mathbf{u}}_0 = (1, 0)^T$, and $T_{-x} \widetilde{\mathbf{u}}_1 = (0, 1)^T$. That is, the set $\begin{pmatrix}1 & -x \\ 0 & 1\end{pmatrix} \begin{pmatrix}a^{-1} & 0 \\ 0 & a\end{pmatrix} B \Lambda_q$ contains both $(1, 0)^T$, and $(0,1)^T$, and so is equal to $\Lambda_q$. From this follows that
\begin{eqnarray*}
B\Lambda_q &=& \begin{pmatrix}a^{-1} & 0 \\ 0 & a\end{pmatrix}^{-1} \begin{pmatrix}1 & -x \\ 0 & 1\end{pmatrix}^{-1} \Lambda_q \\
  &=& \begin{pmatrix}a & ax \\ 0 & a^{-1}\end{pmatrix} \Lambda_q \\
  &=& g_{a, ax} \Lambda_q,
\end{eqnarray*}
and taking $b = ax$ proves the claim.
\end{proof}

We now proceed to prove \cref{theorem: G_q BCZ maps}.

\begin{proof}
We first derive the explicit values of the roof function $R_q(a, b)$ and BCZ map $\BCZ_q(a, b)$ in the second half of the third claim for a given point $(a, b) \in \mathscr{T}_i^q$, $i = 2, 3, \cdots, q - 1$, assuming the remainder of the theorem, and then resume the proof of the theorem from the beginning.

If $(a, b) \in \mathscr{T}_i^q$, with $2 \leq i \leq q - 1$, then $g_{a, b} \mathfrak{w}_i^q$ has the smallest positive slope in $g_{a, b} \Lambda_q \cap S_1$ by our (yet to be proven) assumption. This gives
\begin{eqnarray*}
R_q(a, b) &=& \slope\left(g_{a,b} \mathfrak{w}_i^q\right) \\
 &=& \slope\left(\begin{pmatrix}a & b \\ 0 & a^{-1}\end{pmatrix} \begin{pmatrix}x_i^q \\ y_i^q\end{pmatrix}\right) \\
 &=& \frac{y_i^q}{a(a x_i^q + b y_i^q)} \\
 &=& \frac{y_i^q}{a \times (a, b)^T \cdot \mathfrak{w}_i^q}.
 \end{eqnarray*}
Now, let $[\mathfrak{w}_i^q\ \mathfrak{w}_{i+1}^q]$ be the matrix whose columns are $\mathfrak{w}_i^q$ and $\mathfrak{w}_{i + 1}^q$. The matrix $[\mathfrak{w}_i^q\ \mathfrak{w}_{i+1}^q]$ is in $G_q$ by \cref{proposition: column - unimodular pair identification} since its columns are two unimodular elements of $\Lambda_q$. We show that $h_{R_q(a, b)} g_{a, b} [\mathfrak{w}_i^q\ \mathfrak{w}_{i+1}^q] = g_{(a, b)^T \cdot \mathfrak{w}_i^q, (a, b)^T \cdot \mathfrak{w}_{i + 1}^q}$, and follow that by finding the representative of $g_{(a, b)^T \cdot \mathfrak{w}_i^q, (a, b)^T \cdot \mathfrak{w}_{i + 1}^q}\Lambda_q$ in $\mathscr{T}^q$. (Note that $h_{R_q(a,b)}g_{a,b}\Lambda_q = g_{(a, b)^T \cdot \mathfrak{w}_i^q, (a, b)^T \cdot \mathfrak{w}_{i + 1}^q}\Lambda_q$ by \cref{lemma: orbit-stabilizer lemma}.) Keeping the Farey neighbor identity $\mathfrak{w}_i^q \wedge \mathfrak{w}_{i + 1}^q = x_i^q y_{i + 1}^q - x_{i + 1}^q y_i^q = 1$ in mind, we have
\begin{eqnarray*}
h_{R_q(a, b)} g_{a, b} [\mathfrak{w}_i^q\ \mathfrak{w}_{i+1}^q] &=& \begin{pmatrix}1 & 0 \\ -\frac{y_i^q}{a(a x_i^q + b y_i^q)} & 1\end{pmatrix} \begin{pmatrix} a x_i^q + b y_i^q & a x_{i + 1}^q + b y_{i + 1}^q \\ a^{-1}y_i^q & a^{-1}y_{i + 1}^q\end{pmatrix} \\
 &=& \begin{pmatrix}a x_i^q + b y_i^q & a x_{i + 1}^q + b y_{i + 1}^q \\ 0 & \frac{1}{a x_i^q + b y_i^q}\end{pmatrix} \\
 &=& g_{a x_i^q + b y_i^q, a x_{i + 1}^q + b y_{i + 1}^q} \\
 &=& g_{(a, b)^T \cdot \mathfrak{w}_i^q, (a, b)^T \cdot \mathfrak{w}_{i + 1}^q}.
\end{eqnarray*}
Write $\alpha = (a, b)^T \cdot \mathfrak{w}_i^q$, and $\beta = (a, b)^T \cdot \mathfrak{w}_{i + 1}^a$. Since $T_q = \begin{pmatrix}1 & \lambda_q \\ 0 & 1\end{pmatrix}$ is in $G_q$, and $g_{\alpha, \beta} T_q^k = g_{\alpha, \beta + k \lambda_q \alpha}$ for all $k \in \ZZ$, then $g_{\alpha, \beta} \Lambda_q = g_{\alpha, \beta + k \lambda_q \alpha} \Lambda$ for all $k \in \ZZ$ by \cref{lemma: orbit-stabilizer lemma}. Taking $k_0 = \left\lfloor \frac{1 - \beta}{\lambda_q \alpha} \right\rfloor$, we get $1 - \lambda_q \alpha < \beta + k_0 \lambda_q \alpha \leq 1$. We will also see in a bit that $0 < \alpha \leq 1$ (which is equivalent to $(a, b)$ lying between the lines $\mathcal{L}_i^{q,\text{bot}}$ and $\mathcal{L}_i^{q, \text{top}}$ that we will be working with for the remainder of the proof). We thus have $h_{R_q(a, b)} g_{a, b} \Lambda_q = g_{\alpha, \beta + k_0 \lambda_q \alpha} \Lambda_q$, with $(\alpha, \beta + k_0 \lambda_q \alpha) \in \mathscr{T}^q$, and so
\begin{eqnarray*}
\BCZ_q(a, b) &=& (\alpha, \beta + k_0 \lambda_q \alpha) \\
 &=& \left((a,b)^T \cdot \mathfrak{w}_i^q, (a,b)^T \cdot \mathfrak{w}_{i+1}^q + k_i^q(a, b) \times \lambda_q \times (a, b)^T \cdot \mathfrak{w}_i^q\right)
\end{eqnarray*}
as required.

Now, for the first claim of the theorem: If $B \Lambda_q$ has a horizontal vector of length $a \in (0, 1]$, then there exists $b \in \RR$ such that $B \Lambda_q = g_{a, b} \Lambda_q$ by \cref{lemma: orbit-stabilizer lemma}. Since $T_q \in G_q$, and $g_{a,b}T_q^{\pm 1} = g_{a, b \pm \lambda_q a}$, then $B\Lambda_q = g_{a, b}\Lambda_q = g_{a, b + n \lambda_q a} \Lambda_q$ for all $n \in \ZZ$. From this follows that $B\Lambda_q = g_{(a_B, b_B)}\Lambda_q$, with $(a_B, b_B) = \left(a, b + \left\lfloor\frac{1-b}{\lambda_q a}\right\rfloor \lambda_q a\right) \in \mathscr{T}^q$ as required. It now remains to show that this identification is unique. That is, given $(a, b), (c, d) \in \mathscr{T}^q$, if $g_{a,b}\Lambda_q = g_{c, d}\Lambda_q$, then $(a, b) = (c, d)$. Now,
\begin{equation*}
g_{c,d}^{-1} g_{a,b} = \begin{pmatrix}a/c & b/c - d/a \\ 0 & c/a\end{pmatrix} \in G_q.
\end{equation*}
By the identification in \cref{proposition: column - unimodular pair identification}, we thus have $(a/c, 0)^T = g_{c,d}^{-1} g_{a,b}(1, 0)^T \in \Lambda_q$, and so $a/c = \pm 1$, from which $a = c$. We also have $(b/c - d/a, 1)^T = g_{c,d}^{-1} g_{a,b}(0, 1)^T \in \Lambda_q$. It can be easily seen from the second claim in \cref{proposition: column - unimodular pair identification} that all the points in $\Lambda_q$ at height $y = 1$ are of the form $(n\lambda_q, 1)^T = T_q^n (0, 1)^T$ with $n \in \ZZ$, and so $b/c - d/a = n\lambda_q$ for some $n_0 \in \ZZ$. That is, $b - d = n_0 \lambda_q a$. At the same time, $b, d \in (1 - \lambda_q a, 1]$, and so, since $a > 0$, we get that $b - d \in (-\lambda_q, \lambda_q)$. It now follows that $n_0 = 0$, and $b = d$.

Finally, for the second claim, and the beginning of the third claim of the theorem, we consider the lines
\begin{equation*}
\mathcal{L}_i^{q, \text{bot}} := \{(a, b) \in \RR^2 \mid (a, b)^T \cdot \mathfrak{w}_i^q = 0\},
\end{equation*}
and
\begin{equation*}
\mathcal{L}_i^{q, \text{top}} := \{(a, b) \in \RR^2 \mid (a, b)^T \cdot \mathfrak{w}_i^q = 1\}
\end{equation*}
for $i = 1, 2, \cdots, q - 1$. Note that the lines $\mathcal{L}_1^{q, \text{top}}$ and $\mathcal{L}_{q - 1}^{q, \text{top}}$ agree with the sides $\lambda_q a + b = 1$ and $b = 1$ of $\mathscr{T}^q$. We now show that for $i = 2, 3, \cdots, q - 1$, if $(a, b) \in \mathscr{T}^q$ is in $\mathscr{T}_i^q$ (i.e. above the line $\mathcal{L}_{i-1}^{q, \text{top}}$ and below, or on the line $\mathcal{L}_i^{q, \text{top}}$), then $g_{a, b} \mathfrak{w}_i^q$ belongs to the strip $S_1$, and has the smallest positive slope among the elements of $g_{a, b} \Lambda_q \cap S_1$. For any $i = 2, 3, \cdots, q- 1$, if $(a, b) \in \mathscr{T}^q$ lies in the region above the line $\mathcal{L}_i^{q, \text{bot}}$, and below or on the line $\mathcal{L}_i^{q, \text{top}}$, then the $x$-component $(a, b)^T \cdot \mathfrak{w}_i^q$ of $g_{a, b} \mathfrak{w}_i^q$ satisfies $0 < (a, b)^T \cdot \mathfrak{w}_i^q \leq 1$, and so $g_{a, b} \mathfrak{w}_i^q$ belongs to $g_{a, b}\Lambda_q \cap S_1$. As we will see in a bit, the regions $\mathscr{T}_i^q$, $i = 2, 3, \cdots, q - 1$, cover $\mathscr{T}^q$, and so $g_{a, b} \Lambda_q \cap S_1 \neq \emptyset$ for all $(a, b) \in \mathscr{T}^q$. Moreover, for any $(a, b) \in \mathscr{T}^q$, the elements of $g_{a, b}\Lambda_q \cap S_1$ do not accumulate by the discreteness of $\Lambda_q$, and so there must exist an element of $g_{a, b}\Lambda_q \cap S_1$ with smallest positive slope. This proves the second claim.

Finally, we prove that the regions in question cover the triangle $\mathscr{T}^q$, along with the first half of the third claim of the theorem. I.e., that for $i = 2, 3, \cdots, q - 1$, if $(a, b) \in \mathscr{T}^q$, then $g_{a, b} \mathfrak{w}_i^q$ has the smallest positive slope in $g_{a, b}\Lambda_q$. We break this down into three steps:
\begin{enumerate}
\item For $i = 1, 2, \cdots, q - 1$, the line segments $\mathcal{L}_i^{q, \text{top}} \cap \mathscr{T}^q$ lie above each other, and have increasing (non-positive) slopes. (That is, if $1 \leq i_1 < i_2 \leq q - 1$, then the line segment $\mathcal{L}_{i_1}^{q, \text{top}} \cap \mathscr{T}^q$ lies below the line segment $\mathcal{L}_{i_2}^{q, \text{top}} \cap \mathscr{T}^q$, and $\slope(\mathcal{L}_{i_2}^{q, \text{top}}) > \slope(\mathcal{L}_{i_1}^{q, \text{top}})$.)

\item For each $i = 2, 3, \cdots, q - 1$, the line segment $\mathcal{L}_i^{q, \text{bot}} \cap \mathscr{T}^q$ lies below the line segment $\mathcal{L}_{i - 1}^{q, \text{top}}$. (This proves the claim that the regions $\mathscr{T}_i^q$, $i = 2, 3, \cdots, q - 1$, cover $\mathscr{T}^q$.)

\item For $i = 1, 2, \cdots, q - 2$, if $(a, b) \in \mathscr{T}^q$ lies above the line $\mathcal{L}_i^{q, \text{top}}$, the the $g_{a,b}$ images of $\mathfrak{w}_i^q$ along with its $G_q$-Stern-Brocot children with $\mathfrak{w}_{i+1}^q$ have $x$-components that exceed $1$, and so are not in $g_{a, b} \Lambda_q \cap S_1$.
\end{enumerate}
The third step follows immediately from the fact that the $G_q$-Stern-Brocot children of $\mathfrak{w}_i^q$ and $\mathfrak{w}_{i+1}^q$, $i = 1, 2, \cdots, q - 2$, are all linear combinations of $\mathfrak{w}_i^q$ and $\mathfrak{w}_{i+1}^q$ with coefficients that are at least $1$. It thus remains to prove the first two steps.

For the first step: The lines $\mathcal{L}_i^{q, \text{top}}$, $i = 2, 3, \cdots, q - 1$, intersect the right side $a = 1$ of the Farey triangle $\mathscr{T}^q$ at $(1, b_i)^T$, where $b_i = \frac{1 - x_i}{y_i}$ (recall that $\mathfrak{w}_i^q = (x_i, y_i)^T$). It is easy to see that the heights $b_i$ increase as $i$ increases. (For instance, by acting on the vectors $\{\mathfrak{w}_i^q\}_{i=2}^{q-1}$, which go around the ellipse $Q_q((x, y)^T) = x^2 - \lambda_q x y + y^2 = 1$, by the linear function $T : (x, y)^T \mapsto (1 - x, y)^T$, and considering the inverse slopes of the images.) It now suffices to show that for $i = 2, 3, \cdots, q - 2$, the lines $\mathcal{L}_i^{q, \text{top}}$ and $\mathcal{L}_{i+1}^{q, \text{top}}$ intersect below on the left side $\lambda_q a + b = 1$ of the triangle $\mathscr{T}^q$ to show that the segment $\mathcal{L}_{i + 1}^{q, \text{top}} \cap \mathscr{T}^q$ lies entirely above the segment $\mathcal{L}_i^{q, \text{top}} \cap \mathscr{T}^q$, and that the former has a bigger slope than the latter. (Recall that the side $\lambda_q a + b = 1$ does \emph{not} belong to the set $\mathscr{T}^q$.) To find the sought for intersection, we solve the simultaneous system of equations $(a_0, b_0) \cdot \mathfrak{w}_i^q = 1$ and $(a_0, b_0) \cdot \mathfrak{w}_{i+1}^q = 1$, or equivalently $\begin{pmatrix}x_i^q & y_i^q \\ x_{i + 1}^q & y_{i + 1}^q\end{pmatrix} \begin{pmatrix}a_0 \\ b_0\end{pmatrix} = 1$, for $(a_0, b_0)^T$. Since $\mathfrak{w}_i^q \wedge \mathfrak{w}_{i+1}^q = x_i^q y_{i+1}^q - x_{i+1}^q y_i^q = 1$, we have $\begin{pmatrix}a_o \\ b_0\end{pmatrix} = \begin{pmatrix}y_{i+1}^q & -y_i^q \\ -x_{i+1}^q & x_i^q\end{pmatrix} \begin{pmatrix}1 \\ 1\end{pmatrix} = \begin{pmatrix}y_{i + 1}^q - y_i^q \\ -x_{i+1}^q + x_i^q\end{pmatrix}$. Recalling that $(x_{i+1}^q, y_{i+1}^q)^T = \mathfrak{w}_{i+1}^q = U_q \mathfrak{w}_i^q = (\lambda_q x_i^q - y_i^q, x_i^q)^T$, we have
\begin{eqnarray*}
\lambda_q a_0 + b_0 &=& \lambda_q (y_{i+1}^q - y_i^q) + (-x_{i+1}^q + x_i^q) \\
 &=& \lambda_q (x_i^q - y_i^q) + (-\lambda_q x_i^q + y_i^q + x_i^q) \\
 &=& (1 - \lambda_q) y_i^q + x_i^q.
\end{eqnarray*}
The intersection $(a_0, b_0)^T$ thus lies on or below $\lambda_q a + b = 1$ if $(1 - \lambda_q)y_i^q + x_i^q \leq 1$, or equivalently $\frac{x_i^q - 1}{\lambda_q - 1} \leq y_i^q$. (Recall that $\lambda_q = 2 \cos(\pi/q) \geq 1$ for $q \geq 3$.) Now we consider the ellipse $x^2 - \lambda_q x y + y^2 = 1$ and the line $y = \frac{x - 1}{\lambda_q - 1}$. The two points $\mathfrak{w}_0^q = (1, 0)^T$ and $\mathfrak{w}_1^q = (\lambda_q, 1)^T$ lie at the intersection of the aforementioned ellipse and line, and so the remaining points $\{\mathfrak{w}_i^q\}_{i=2}^{q-1}$ lie above the line $y = \frac{x - 1}{\lambda_q - 1}$, thus proving the inequality $y_i \geq \frac{x_i - 1}{\lambda_q - 1}$ for all $i = 2, 3, \cdots, q - 2$.

For the second step: If $i = 2, 3, \cdots, q - 1$, the slope of the line segment $\mathcal{L}_{i}^{q, \text{bot}}$ agrees with that of $\mathcal{L}_{i}^{q, \text{top}}$, and so exceeds that of $\mathcal{L}_{i-1}^{q, \text{top}}$. It thus suffices to show that for $i = 2, 3, \cdots, q - 1$, the lines $\mathcal{L}_i^{q, \text{bot}}$ and $\mathcal{L}_{i - 1}^{q, \text{top}}$ intersect at a point on the right of the side $a = 1$ of the triangle $\mathscr{T}^q$. Towards that end, we compare the heights $b_{i-1} = \frac{1-x_{i-1}^q}{y_{i-1}^q}$ and $b_i^\prime = \frac{-x_1^q}{y_1^q}$ at which the lines $\mathcal{L}_{i-1}^{q, \text{top}}$ and $\mathcal{L}_i^{q, \text{bot}}$ intersect the side $a = 1$ of $\mathcal{T}^q$. We have that $b_{i-1} \geq b_i^\prime$ if and only if $y_i^q \geq x_{i-1}^q y_i^q - x_i^q y_{i-1}^q = \mathfrak{w}_{i-1}^q \wedge \mathfrak{w}_i^q = 1$, which is true for all $i = 2, 3, \cdots, q - 1$ by \cref{proposition: column - unimodular pair identification}. This ends the proof.
\end{proof}

\subsection{The $h_\cdot$-Periodic Points in $\SL(2, \RR)/G_q$, and the $\BCZ_q$-Periodic Points in $\mathscr{T}^q$}

\begin{lemma}
\label{lemma: vertical vectors and periodicity}
For any $A \in \SL(2, \RR)$, the following are equivalent.
\begin{enumerate}
\item The set $A\Lambda_q$ contains a vertical vector.
\item There exists $s_0 > 0$ such that $h_{s_0}(A\Lambda_q) = A\Lambda_q$. That is, $A\Lambda_q$ is $h_\cdot$-periodic.
\item There exists $\tau_0 > 0$ such that $A\Lambda_q \cap S_{\tau_0} = \emptyset$.
\end{enumerate}
Moreover, if $A\Lambda_q$ contains a vertical vector of length $a$, then the $h_\cdot$-period of $A\Lambda_q$ is $\lambda_q a^2$.
\end{lemma}

\begin{proof}
We prove $(1) \Rightarrow (2) \Rightarrow (3)$ directly, and $(3) \Rightarrow (1)$ by contradiction.

First, we note that $\Lambda_q$ is $h_\cdot$-periodic since $h_\lambda \in G_q$, and so $h_\lambda \Lambda_q = \Lambda_q$. We also note that for any $s, t \in \RR$, $\tau > 0$, and $B \in \SL(2, \RR)$ we have $h_t (s_\tau h_s B\Lambda_q) =s_\tau h_s (h_{t\tau^2} B\Lambda_q)$, and so $B\Lambda_q$ is $h_\cdot$-periodic iff $s_\tau h_s B\Lambda_q$ is $h_\cdot$-periodic.

For $(1) \Rightarrow (2)$: Let $A\Lambda_q$ contain a vertical vector $(0, a)^T$, with $a > 0$. Then $s_a A\Lambda_q$ contains the vertical vector $(0, 1)^T$. Pick any vector $\mathbf{u}_0 \in s_a A\Lambda_q$ such that $\mathbf{u}_0 \wedge (0, 1)^T = 1$ (and so the $x$-component of $\mathbf{u}_0$ is $1$). If $s = \slope(\mathbf{u}_0) \neq 0$, then $h_s\mathbf{u}_0$ is a horizontal vector with the same $x$-component as $\mathbf{u}_0$, i.e. $1$, and $h_s (0, 1)^T = (0, 1)^T$. By \cref{lemma: orbit-stabilizer lemma}, $h_s s_a A\Lambda_q = \Lambda_q$, and so $A\Lambda_q = s_\frac{1}{a} h_{-s} \Lambda_q$, from which $A\Lambda_q$ is $h_\cdot$-periodic.

For $(2) \Rightarrow (3)$: Let $\tau_1 > 0$ be such that $A\Lambda_q \cap S_{\tau_1} \neq \emptyset$. For any vector $\mathbf{u}_0 \in A\Lambda_q \cap S_{\tau_1}$, and any $s > 0$, the vector $h_s \mathbf{u}_0$ has the same $x$-component as $\mathbf{u}_0$, and $\slope(h_s \mathbf{u}_0) = \slope(\mathbf{u}_0) - s$. If $s_0$ is an $h_\cdot$-period of $A\Lambda_q$, then the set of lengths of the finitely many horizontal vectors that appear in $h_s(A\Lambda_q \cap S_{\tau_1})$ as $s$ goes from $0$ to $s_0$ agrees with the set of $x$-components of the vectors in $A\Lambda_q \cap S_{\tau_1}$. This implies that the $X$-components of vectors in $A\Lambda_q$ are bounded from below, and so there must exist a $\tau_0 > 0$ such that $A\Lambda_q \cap S_{\tau_0} = \emptyset$.

Finally, we prove $(3) \Rightarrow (1)$ by contradiction. If $A\Lambda_q$ contains no vertical vectors, then $\mathbf{v} = (0, 1)^T$ is not parallel to any vector in $A\Lambda_q$. By \cref{corollary: odds and ends}, there exists sequences of unimodular pairs $\{\mathbf{u}_{0, n}, \mathbf{u}_{1, n}\}_{n=1}^\infty$ such that for each $n \geq 2$, the vector $\mathbf{v}$ belongs to the sector $(0, \infty)\mathbf{u}_{0, n} + (0, \infty)\mathbf{u}_{1, n}$, and $\mathbf{u}_{0, n}, \mathbf{u}_{1, n}$ are $G_q$-Stern-Brocot children of $\mathbf{u}_{0, n - 1}, \mathbf{u}_{1, n - 1}$. From \cref{equation: sin inequalities}, we get
\begin{equation*}
0 < \mathbf{u}_{0, n} \wedge (0, 1)^T < \frac{1}{\|\mathbf{u}_{1, n}\|} \to 0.
\end{equation*}
That is, $A\Lambda_q$ contains vectors with arbitrarily small positive $x$-components.

Finally, if $A\Lambda_q$ contains a vertical vector of length $a$, we showed earlier in this proof that $A\Lambda_q$ must be of the form $s_\frac{1}{a} h_{-s} \Lambda_q$ for some $s \in \RR$. For any $t \in \RR$, we have $h_t\left(s_\frac{1}{a} h_{-s} \Lambda_q\right) = s_\frac{1}{a} h_{-s} \left(h_\frac{t}{a^2} \Lambda_q\right)$, which implies that the $h_\cdot$-period of $A\Lambda_q$ is $a^2$ times that of $\Lambda_q$.
\end{proof}

\begin{corollary}
\label{corollary: characterizing BCZ periodic points}
For any $(a, b) \in \mathscr{T}^q$, the following are equivalent.
\begin{enumerate}
\item The point $(a, b)$ is $\BCZ_q$-periodic.
\item The set $g_{a, b}\Lambda_q$ is $h_\cdot$-periodic.
\item The ratio $b/a$ is the (inverse) slope of a vector in $\Lambda_q$.
\end{enumerate}
\end{corollary}

\begin{proof}
That the first two claims are equivalent is obvious, and so we proceed to characterize the points $(a, b) \in \mathscr{T}^q$ for which $g_{a, b}\Lambda_q$ is $h_\cdot$-periodic.

Note that $g_{a, b} = s_a g_{1, b/a}$, and that for any $s \in \RR$, $h_s (s_a g_{1, b/a}) = s_a (h_{a^2 s} g_{1, b/a})$. That is, $g_{a, b} \Lambda_q$ is $h_\cdot$-periodic if $g_{1, b/a}\Lambda_q$ is $h_\cdot$-periodic. By \cref{lemma: vertical vectors and periodicity}, $g_{1,b/a}\Lambda_q$ is $h_\cdot$-periodic iff it contains a vertical vector. Since $g_{1, b/a}$ is a horizontal shear, the set $g_{1,b/a}\Lambda_q$ contains a vertical vector exactly when $b/a$ is the inverse slope of a vector in $\Lambda_q$. The claim now follows from the symmetry of $\Lambda_q$ against the line $y = x$.
\end{proof}

\subsection{The $G_q$-Next-Term Algorithm}
\label{subsection: G_q next term algorithm}

\begin{theorem}
\label{theorem: next term algorithm}
Let $A \in \SL(2, \RR)$, $\tau > 0$ be such that $A \Lambda_q \cap S_\tau \neq \emptyset$, and $\{\mathbf{u}_n = (q_n, a_n)^T\}_{n=0}^\infty$ be elements of $A \Lambda_q \cap S_\tau$ with successive slopes. The set $s_\frac{1}{\tau} h_{\slope(\mathbf{u}_0)} A \Lambda_q$ has a horizontal vector $(q_0/\tau, 0)^T \in s_\frac{1}{\tau} h_{\slope(\mathbf{u}_0)} A \Lambda_q \cap S_1$, and hence corresponds to a unique point $(a, b) \in \mathscr{T}^q$ (i.e. $s_\frac{1}{\tau} h_{\slope(\mathbf{u}_0)} A \Lambda_q = g_{a,b}\Lambda_q$). The following are then true.

\begin{enumerate}

\item For each $n \geq 0$, the set $s_\frac{1}{\tau} h_{\slope(\mathbf{u}_n)} A \Lambda_q$ has a horizontal vector $(q_n/\tau, 0)^T \in s_\frac{1}{\tau} h_{\slope(\mathbf{u}_n)} A \Lambda_q \cap S_1$, and corresponds to $\BCZ_q^n(a, b)$.

\item If we denote the $x$-component of $\BCZ^n(a, b)$ by $L_n^q(a, b)$ for all $n \geq 0$, then the $x$-components of the vectors $\{\mathbf{u}_n\}_{n=0}^\infty$ are equal to
\begin{equation*}
q_n = \tau L_n^q(a, b) = \tau L_0^q(\BCZ_q^n(a, b)).
\end{equation*}
Moreover, the $y$-components of the vectors $\{\mathbf{u}_n\}_{n=0}^\infty$ can be recursively generated using the formula
\begin{equation*}
a_{n + 1} = q_{n + 1} \left(\frac{a_n}{q_n} + \frac{1}{\tau^2} R_q(\BCZ_q^n(a, b))\right)
\end{equation*}
for all $n \geq 0$.
\end{enumerate}
\end{theorem}

This motivates the following definition.

\begin{definition}
For any $A \in \SL(2, \RR)$, $\tau > 0$ with $A\Lambda_q \cap S_\tau \neq \emptyset$, and $\mathbf{u} \in A \Lambda_q \cap S_\tau$, we refer to the unique point in the Farey triangle $\mathscr{T}^q$ corresponding to $s_\frac{1}{\tau} h_{\slope(\mathbf{u})} A \Lambda_q$ from \cref{theorem: next term algorithm} as the \emph{$G_q$-Farey triangle representatitve} of the triple $(A, \tau, \mathbf{u})$, and denote it by $\FTR_q(A, \tau, \mathbf{u})$.
\end{definition}

Using this notation, we can succinctly rewrite the first claim in \cref{theorem: next term algorithm}
as
\begin{equation*}
\FTR_q(A, \tau, \mathbf{u}_n) = \BCZ_q^n\left(\FTR_q(A, \tau, \mathbf{u}_0)\right)
\end{equation*}
for all $n \geq 0$.

\begin{remark}
\label{remark: practical next term algorithm}
For any $\tau \geq 1$, the vector $\mathbf{u}_0 = (1, 0)^T$ belongs to $\Lambda_q \cap S_\tau$, and so $\FTR_q(I_2, \tau, \mathbf{u}_0)$ is well defined. We have
\begin{equation*}
s_\frac{1}{\tau} h_{\slope(\mathbf{u_0})} \Lambda_q = s_\frac{1}{\tau} \Lambda_q = s_\frac{1}{\tau} T_q^{\left\lfloor\frac{\tau}{\lambda_q}\right\rfloor} \Lambda_q = g_{\frac{1}{\tau}, \left\lfloor\frac{\tau}{\lambda_q}\right\rfloor\frac{\lambda_q}{\tau}}\Lambda_q,
\end{equation*}
with $0 < \frac{1}{\tau} \leq 1$, and $1 - \lambda_q \left( \frac{1}{\tau}\right) < \left\lfloor\frac{\tau}{\lambda_q}\right\rfloor\frac{\lambda_q}{\tau} \leq 1$, from which $\left(\frac{1}{\tau}, \left\lfloor\frac{\tau}{\lambda_q}\right\rfloor\frac{\lambda_q}{\tau}\right) \in \mathscr{T}^q$ is the $G_q$-Farey triangle representative of the triple $(\Lambda_q, \tau, \mathbf{u}_0)$. By the symmetry of $\Lambda_q$ against the lines $y = \pm x$, $x = 0$, and $y = 0$, it suffices to generate the vectors in $\Lambda_q \cap S_\tau$ with slopes in $[0, 1]$ to get all the vectors in $\Lambda_q \cap [-\tau, \tau]^2$.
\end{remark}

\begin{proof}[Proof of \cref{theorem: next term algorithm}]
For each $n \geq 0$, a direct calculation gives $s_\frac{1}{\tau} h_{\slope(\mathbf{u}_n)} \mathbf{u}_n = (q_n/\tau, 0)$, which is a horizontal vector of length not exceeding $1$ in $s_\frac{1}{\tau} h_{\slope(\mathbf{u}_n)} A\Lambda_q$. By the first claim in \cref{theorem: G_q BCZ maps}, the set $s_\frac{1}{\tau} h_{\slope(\mathbf{u}_n)} A\Lambda_q$ corresponds to a unique point $(c_n, d_n) \in \mathscr{T}^q$ with $q_n/\tau = a_n$, and $(c_0, d_0) = (a, b)$. The vectors $\mathbf{u}_n$ and $\mathbf{u}_{n+1}$ have consecutive slopes in $A\Lambda_q \cap S_\tau$, and so the two vectors $s_\frac{1}{\tau} h_{\slope(\mathbf{u}_n)} \mathbf{u}_n$ and $s_\frac{1}{\tau} h_{\slope(\mathbf{u}_n)} \mathbf{u}_{n+1}$ have consecutive slopes in $s_\frac{1}{\tau} h_{\slope(\mathbf{u}_n)} A\Lambda_q \cap S_1$. In other words, the vector $s_\frac{1}{\tau} h_{\slope(\mathbf{u}_n)} \mathbf{u}_{n+1}$ is the vector of smallest positive slope in $g_{c_n, d_n} \Lambda_q \cap S_1$, from which
\begin{eqnarray*}
R_q(c_n, d_n) &=& \slope\left(s_\frac{1}{\tau} h_{\slope(\mathbf{u}_n)} \mathbf{u}_{n+1}\right) \\
 &=& \tau^2 \left(\slope(\mathbf{u}_{n+1}) - \slope(\mathbf{u}_n)\right) \\
 &=& \tau^2 \left(\frac{a_{n+1}}{q_{n+1}} - \frac{a_n}{q_n}\right),
\end{eqnarray*}
and
\begin{equation*}
h_{R_q(c_n, d_n)} g_{c_n, d_n}\Lambda_q = s_\frac{1}{\tau} h_{\slope(\mathbf{u}_{n+1})} A\Lambda_q = g_{c_{n+1}, d_{n+1}}\Lambda_q,
\end{equation*}
and so
\begin{equation*}
\BCZ_q(c_n, d_n) = (c_{n+1}, d_{n+1})
\end{equation*}
by the second claim in \cref{theorem: G_q BCZ maps}. By induction, we get $(c_n, d_n) = \BCZ_q^n(c_0, d_0)$, $q_n = \tau c_n = \tau L_n^q(c_0, d_0)$, and the sought for recursive expression for $a_{n+1}$.

\end{proof}

\section{A Poincar\'{e} Cross Section for the Horocycle Flow on the Quotient $\SL(2, \RR)/G_q$}
\label{section: cross section}

Let $X_q$ be the homogeneous space $\SL(2, \RR)/G_q$, $\mu_q$ be the probability Haar measure on $X_q$ (i.e. $\mu_q(X_q) = 1$), and $\Omega_q$ be the subset of $X_q$ corresponding to sets $A\Lambda_q$, $A \in \SL(2, \RR)$, with a horizontal vector of length not exceeding $1$. Note that $\Omega_q$ can be identified with the Farey triangle $\mathscr{T}^q$ via $\left((a, b) \in \mathscr{T}^q\right) \mapsto (g_{a,b}G_q \in \Omega_q)$ by \cref{lemma: orbit-stabilizer lemma} and \cref{theorem: G_q BCZ maps}. Finally, let $m_q = \frac{2}{\lambda_q}dadb$ be the Lebesgue probability measure on $\mathscr{T}^q$. Following \cite{Athreya2013-ql}, we have the following.

\begin{theorem}
\label{theorem: cross section}
The triple $(\mathscr{T}^q, m_q, \BCZ_q)$, with $\mathscr{T}^q$ identified with $\Omega_q$, is a cross section to $(X_q, \mu_q, h_\cdot)$, with roof function $R_q$.
\end{theorem}

\begin{proof}
Consider the suspension space
\begin{equation*}
S_{R_q}\mathscr{T}^q := \{\left((a, b), s\right) \in \mathscr{T}^q \times \RR \mid 0 \leq s \leq R_q(a, b)\} / \sim_q,
\end{equation*}
with $\left((a, b), R_q(a, b)\right) \sim_q \left(\BCZ_q(a, b), 0\right)$ for all $(a, b) \in \mathscr{T}^q$, as a subset of $X_q$. The suspension flow of $S_{R_q}\mathscr{T}^q$ can be identified with the horocycle flow $h_\cdot$ on $S_{R_q}\mathscr{T}^q$ as a subset of $X_q$ by \cref{theorem: G_q BCZ maps}. The probability measure $dm_q^{R_q} = \frac{1}{m_q(R_q)} dm_qds$ is $h_\cdot$-invariant, and the suspension space $S_{R_q}\mathscr{T}^q$ contains non-closed horocycles (e.g by \cref{lemma: vertical vectors and periodicity} and \cref{corollary: characterizing BCZ periodic points}). By Dani-Smillie \cite{Dani1984-ji}, the subset $S_{R_q}\mathscr{T}^q$ has full $\mu_q$ measure in $X_q$, and the probability measures $dm_q^{R_q}$ and $\mu_q$ can be identified. This proves the claim.

\end{proof}

\subsection{Limiting Distributions of Farey Triangle Representatives, and Equidistribution of the Slopes of $\Lambda_q$}
\label{subsection: limiting distributions of FTR}

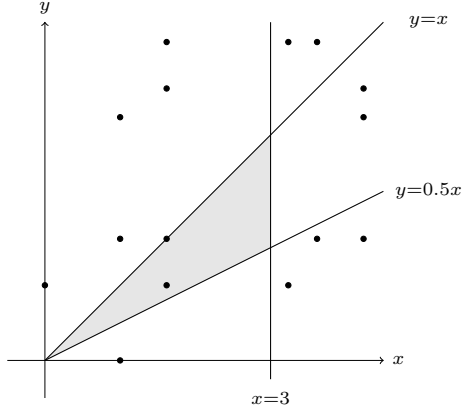
\begin{figure}
\centering
\begin{tikzpicture}
\draw[->,ultra thin] (-0.5,0)--(4.5,0) node[right]{$\scriptstyle x$};
\draw[->,ultra thin] (0,-0.5)--(0,4.5) node[above]{$\scriptstyle y$};
\fill[fill=gray!20] (0,0) -- (3,1.5) -- (3,3);
\draw [fill=black] (1, 0) circle [radius=0.1em];
\draw [fill=black] (3.236, 1) circle [radius=0.1em];
\draw [fill=black] (4.236, 1.618) circle [radius=0.1em];
\draw [fill=black] (3.618, 1.618) circle [radius=0.1em];
\draw [fill=black] (1.618, 1) circle [radius=0.1em];
\draw [fill=black] (4.236, 3.236) circle [radius=0.1em];
\draw [fill=black] (4.236, 3.618) circle [radius=0.1em];
\draw [fill=black] (1.618, 1.618) circle [radius=0.1em];
\draw [fill=black] (0, 1) circle [radius=0.1em];
\draw [fill=black] (1, 3.236) circle [radius=0.1em];
\draw [fill=black] (1.618, 4.236) circle [radius=0.1em];
\draw [fill=black] (1.618, 3.618) circle [radius=0.1em];
\draw [fill=black] (1, 1.618) circle [radius=0.1em];
\draw [fill=black] (3.236, 4.236) circle [radius=0.1em];
\draw [fill=black] (3.618, 4.236) circle [radius=0.1em];
\draw (3, -0.25) -- (3, 4.5);
\draw (0, 0) -- (4.5, 2.25);
\draw (0, 0) -- (4.5, 4.5);
\node at (3, -0.5) {$\scriptstyle x = 3$};
\node at (5.1, 2.25) {$\scriptstyle y = 0.5x$};
\node at (5.1, 4.5) {$\scriptstyle y = x$};
\end{tikzpicture}
\caption{The set $\mathcal{F}_{[0.5, 1]}(\Lambda_5, 3)$ is the collection of points from $\Lambda_5$ inside the shaded region bounded by the lines $y = 0.5x$, $y = x$, and $x  = 3$.}
\end{figure}

For any $A \in \SL(2, \RR)$, $\tau > 0$, and interval $I \subseteq \RR$, we denote by
\begin{equation*}
\mathcal{F}_I(A\Lambda_q, \tau) := \{\mathbf{u} \in A\Lambda_q \cap S_\tau \mid \slope(\mathbf{u}) \in I\}
\end{equation*}
the set of vectors in $A\Lambda_q$ with positive $x$-components not exceeding $\tau$, and slopes in $I$. If $I \subset \RR$ is a finite interval, we write
\begin{equation*}
N_I(A\Lambda_q, \tau) := \# \mathcal{F}_I(A\Lambda_q, \tau)
\end{equation*}
for the number of elements of $\mathcal{F}_I(A\Lambda_q, \tau)$. Note that if $I$ is a non-degenerate interval, then $\lim_{\tau \to \infty} N_I(A\Lambda_q, \tau) = \infty$ by the density of the slopes of $A\Lambda_q$ from \cref{corollary: odds and ends}.

For any $A \in \SL(2, \RR)$, finite, non-empty, non-degenerate interval $I \subset \RR$, and $\tau > 0$ with $\mathcal{F}_I(A\Lambda_q, \tau) = \{\mathbf{u}_i\}_{i=0}^{N_I(A\Lambda_q, \tau) - 1} \neq \emptyset$, we define the following probability measure on the Farey triangle $\mathscr{T}^q$
\begin{eqnarray*}
\rho_{A\Lambda_q, I, \tau} &:=& \frac{1}{N_I(A\Lambda_q, \tau)} \sum_{i = 0}^{N_I(A\Lambda_q, \tau) - 1} \delta_{\FTR_q(A, \tau, \mathbf{u}_i)} \\
 &=& \frac{1}{N_I(A\Lambda_q, \tau)} \sum_{i = 0}^{N_I(A\Lambda_q, \tau) - 1} \delta_{\BCZ_q^i\left(\FTR_q(A, \tau, \mathbf{u}_0)\right)}.
\end{eqnarray*}

\begin{theorem}
\label{theorem: weak limit and asymptotic growth}
Let $A \in \SL(2, \RR)$, and $I \subset \RR$ be a finite, non-empty, non-degenerate interval in $\RR$. Then as $\tau \to \infty$, the number of elements of $\mathcal{F}_I(A\Lambda_q, \tau)$ has the asymptotic growth
\begin{equation*}
N_I(A\Lambda_q, \tau) \sim \frac{|I|}{m_q(R_q)} \tau^2,
\end{equation*}
and the measures $\rho_{A\Lambda_q, I, \tau}$ converge weakly
\begin{equation*}
\rho_{A\Lambda_q, I, \tau} \rightharpoonup m_q
\end{equation*}
to the probability Lebesgue measure $m_q$.
\end{theorem}

\begin{corollary}
\label{corollary: equidistribution of slopes}
For any $A \in \SL(2, \RR)$, and any finite interval $\emptyset \neq I \subset \RR$, the slopes of the vectors $\mathcal{F}_I(A\Lambda_q, \tau)$ equidistribute in $I$ as $\tau \to \infty$.
\end{corollary}

\begin{proof}[Proof of \cref{theorem: weak limit and asymptotic growth} and \cref{corollary: equidistribution of slopes}]

For $\tau > 0$ with $\mathcal{F}_I(A\Lambda_q, \tau) \neq \emptyset$, we define the measures
\begin{equation*}
\sigma_{A\Lambda_q, I, \tau} = \frac{N_I(A\Lambda_q, \tau)}{\tau^2} \rho_{A\Lambda_q, I, \tau}
\end{equation*}
on the Farey triangle $\mathscr{T}^q$. Denote by $d\sigma_{A\Lambda_q, I, \tau}^{R_q}$ the measure $d\sigma_{g\Lambda_q, I, \tau}ds$ on the suspension space $S_{R_q}\mathscr{T}^q$ (which can identified with $X_q$ by \cref{theorem: cross section}). In what follows, we denote the elements of $\mathcal{F}_I(A\Lambda_q, \tau)$ by $\{\mathbf{u}_i = \mathbf{u}_i(A\Lambda_q, I, \tau)\}_{i=0}^{N_I(A\Lambda_q, \tau) - 1}$, and write $\mathbf{u}_{N_I(A\Lambda_q, \tau)} = \mathbf{u}_{N_I(A\Lambda_q, \tau)}(A\Lambda_q, I, \tau)$ for the element of $A\Lambda_q \cap S_\tau$ of smallest slope bigger than any value in $I$. By the density of the slopes of $A\Lambda_q$ from \cref{corollary: odds and ends}, we have that $\slope(\mathbf{u}_0)$ and $\slope(\mathbf{u}_{N_I(A\Lambda_q, \tau)})$ converge to the end points of the interval $I$, which we denote $\alpha$ and $\beta$ (i.e. $|I| = \beta - \alpha$). We show the convergence $\sigma_{A\Lambda_q, I, \tau} \rightharpoonup |I|/m_q(R_q) m_q$ by proving the convergence $\sigma_{A\Lambda_q, I, \tau}^{R_q} \rightharpoonup |I| \mu_q$. Given any continuous, bounded function $f : X_q \to \RR$, we have
\begin{eqnarray*}
\sigma_{g\Lambda_q, I, \tau}^{R_q}(f) &=& \frac{1}{\tau^2} \int_{\tau^2 \slope(\mathbf{u}_0)}^{\tau^2 \slope(\mathbf{u}_{N_I(A\Lambda_q, \tau)})} f\left(h_s \left(s_\frac{1}{\tau}AG_q\right)\right)\,ds \\
 &=& \frac{1}{\tau^2} \int_{\tau^2 \slope(\mathbf{u}_0)}^{\tau^2 \slope(\mathbf{u}_{N_I(A\Lambda_q, \tau)})} f\left(s_\frac{1}{\tau} h_\frac{s}{\tau^2}AG_q\right)\,ds \\
 &=& \int_{\slope(\mathbf{u}_0)}^{\slope(\mathbf{u}_{N_I(A\Lambda_q, \tau)})} f\left(s_\frac{1}{\tau} h_t AG_q\right)\,dt \\
 &=& \int_\alpha^\beta f\left(s_\frac{1}{\tau} h_t AG_q\right)\,dt + o(1) \\
 &\to& (b - a) \mu_q(f)
\end{eqnarray*}
as $\tau \to \infty$ (with the convergence of the measures supported on horocycles following from, for example, \cite[2.2.1]{Kleinbock1996-ag}). This proves the weak convergence $\sigma_{g\Lambda_q, I, \tau}^{R_q} \rightharpoonup |I| \mu_q$. Denoting by $\pi_q : S_{R_q}\mathscr{T}^q \to \mathscr{T}^q$ the projection map $\left(\left((a, b), s\right) \in S_{R_q}\mathscr{T}^q\right) \mapsto \left((a, b) \in \mathscr{T}^q\right)$, we thus have
\begin{equation*}
\sigma_{g\Lambda_q, I, \tau} = \frac{1}{R_q} (\pi_q)_\ast \sigma_{g\Lambda_q, I, \tau}^{R_q} \rightharpoonup \frac{|I|}{R_q} (\pi_q)_\ast \mu_q = \frac{|I|}{m_q(R_q)} m_q.
\end{equation*}
From $\rho_{A\Lambda_q, I, \tau}(\mathscr{T}^q) = 1$, we get
\begin{equation*}
\lim_{\tau \to \infty} \frac{N_I(A\Lambda_q, \tau)}{\tau^2} = \lim_{\tau \to \infty} \sigma_{A\Lambda_q, I, \tau}(\mathscr{T}^q) = \frac{|I|}{m_q(R_q)} m_q(\mathscr{T}^q) = \frac{|I|}{m_q(R_q)},
\end{equation*}
which is the asymptotic growth from \cref{theorem: weak limit and asymptotic growth}. This also gives the weak limit $\rho_{A\Lambda_q, I, \tau} \rightharpoonup m_q$.

As for \cref{corollary: equidistribution of slopes}, if $\emptyset \neq J \subseteq I$ is any non-empty subinterval of $I$, we have
\begin{equation*}
\lim_{\tau \to \infty} \frac{N_J(A\Lambda_q, \tau)}{N_I(A\Lambda_q, \tau)} = \frac{|J|}{|I|},
\end{equation*}
which proves the sought for equidistribution.
\end{proof}

\section{Applications}
\label{section: applications}

In this section, we give a few applications of the $G_q$-BCZ maps to the statistics of subsets of $\Lambda_q$. In \cref{subsection: counting in triangles}, we derive the main asymptotic term for the number of vectors of $\Lambda_q$ in homothetic dilations of triangles. In \cref{subsubsection: slopegap distribution}, we derive the distribution of the slope gaps of $\Lambda_q$. Finally, in \cref{subsubsection: centdist distribution}, we derive the distribution of the Euclidean distances between the centers of $G_q$
-Ford circles. Several other applications of the $G_3$-BCZ map to the statistics of the visible lattice points $\Lambda_3 = \ZZ_\text{prim}^2 = \{(x, y) \in \ZZ^2 \mid \gcd(x, y) = 1\}$ can be similarly extended--almost verbatim--to general $\Lambda_q$. This list includes, but is not limited to, an old Diophantine approximation problem of \cite{Erdos1958-xi} Erd\"{o}s, P., Sz\"{u}sz, P., \& Tur\'{a}n solved independently by Xiong and Zaharescu \cite{Xiong2006-ef}, and Boca \cite{Boca2008-bu} for $G_3 = \SL(2, \ZZ)$, and Heersink \cite{Heersink2016-hg} for finite index subgroups of $G_3 = \SL(2, \ZZ)$; the average depth of cusp excursions of the horocycle flow on $X_2 = \SL(2, \RR)/G_3$ by Athreya and Cheung \cite{Athreya2013-ql}; and the statistics of weighted Farey sequences by Panti \cite{Panti2015-jz}.

\subsection{Asymptotic Growth of the Number of Elements of $\Lambda_q$ in Homothetic Dilations of Triangles}
\label{subsection: counting in triangles}

For any $A \in \SL(2, \RR)$, $\tau > 0$, and finite interval $I \subset \RR$, the set $\mathcal{F}_I(A\Lambda_q, \tau)$ introduced in \cref{subsection: limiting distributions of FTR} is the collection of points of $A\Lambda_q$ which belong to the triangle $\{(x, y)^T \in \RR^2 \mid y/x \in I, 0 < x \leq \tau\}$. We have the main term for the asymptotic growth rate of the number of aforementioned vectors $N_I(A\Lambda_q, \tau)$ as $\tau \to \infty$, which can be immediately interpreted as a statement on the asymptotic growth of the number of vectors of $A\Lambda_q$ in homothetic dilations of triangles that have a vertex at the origin as we do in \cref{proposition: counting in triangles}. In \cref{corollary: equidistribution in the square}, we show the equidistribution of the homothetic dilations $\frac{1}{\tau}\Lambda_q$ in the square $[-1, 1]^2$ as $\tau \to \infty$. In what follows, we write $f(\tau) \sim g(\tau)$ as $\tau \to \infty$ for any two functions $f, g$ to indicate that $\lim_{\tau \to \infty} f(\tau)/g(\tau) = 1$.

\begin{proposition}
\label{proposition: counting in triangles}
Let $\Delta$ be a triangle in the plane $\RR^2$ with one vertex at the origin. Then for any $A \in \SL(2, \RR)$, and any $\tau > 0$ the number of elements $\#\left(A\Lambda_q \cap \tau \Delta \right)$ has the asymptotic growth rate
\begin{equation*}
\#\left(A\Lambda_q \cap \tau \Delta\right) \sim \left(\frac{2}{m_q(R_q)} \area(\Delta)\right) \tau^2
\end{equation*}
as $\tau \to \infty$.
\end{proposition}

We also get the following.

\begin{corollary}
\label{corollary: equidistribution in the square}
For any $A \in \SL(2, \RR)$, and $\tau \geq 1$, let $\lambda_\tau^{A\Lambda_q}$ be the probability measure defined for any Borel subset $C$ of the square $[-1, 1]^2$ by
\begin{equation*}
\lambda_\tau^{A\Lambda_q}(C) = \frac{\#(\frac{1}{\tau}A\Lambda_q \cap C)}{\#\left(\frac{1}{\tau}A\Lambda_q \cap [-1, 1]^2\right)}.
\end{equation*}
Then the measures $\lambda_\tau^{A\Lambda_q}$ converge weakly to the Lebesgue probability measure $\Unif_{[-1, 1]^2}$ on $[-1, 1]^2$ as $\tau \to \infty$.
\end{corollary}

\begin{proof}[Proof of \cref{proposition: counting in triangles}]
We first prove the theorem assuming that the side $L$ of the triangle $\Delta$ opposite to the origin is included in $\Delta$. Let $\operatorname{rot}_\Delta \in \SL(2, \RR)$ be the rotation that rotates the side of $\Delta$ opposite to the vertex at the origin onto a vertical line segment. (That is, the side of $\operatorname{rot}_\Delta \Delta$ opposite to the vertex at the origin is vertical.) Denote by $d_\Delta > 0$ the perpendicular distance from the vertex at the origin to the side of $\operatorname{rot}_\Delta \Delta$ opposite to the aforementioned vertex, and by $I_\Delta \subset \RR$ the interval of slopes of the points in $\operatorname{rot}_\Delta \Delta$. For any $\tau > 0$, we have that $\tau (\operatorname{rot}_\Delta \Delta) \cap (\operatorname{rot}_\Delta A \Lambda_q) = \mathcal{F}_I(\operatorname{rot}_\Delta A \Lambda_q, \tau d)$, and that the rotation $\operatorname{rot}_\Delta$ is a bijection from $\tau \Delta \cap A \Lambda_q$ to $\tau (\operatorname{rot}_\Delta \Delta) \cap (\operatorname{rot}_\Delta A \Lambda_q)$. From this and \cref{theorem: weak limit and asymptotic growth} follows that
\begin{equation*}
\#\left(A \Lambda_q \cap \tau \Delta\right) = N_{I_\Delta}(\operatorname{rot}_\Delta A \Lambda_q, \tau d) \sim \frac{|I_\Delta|}{m_q(R_q)} (\tau d_\Delta)^2 = \frac{2}{m_q(R_q)} \area(\Delta) \tau^2
\end{equation*}
which proves the claim.

Including or excluding any of the two sides of the triangle $\Delta$ that pass through the origin does not change $|I_\Delta|$, and hence the main term for the asymptotic growth in question remains the same. We now show that the main term does not change when the side $L$ of $\Delta$ opposite to the origin  is removed as well. For any $\delta > 0$, denote by $\Delta^\prime = \Delta^\prime(\Delta, \delta)$ the homothetic dilation of $\Delta$ such that $0 < \area(\Delta) - \area(\Delta^\prime) \leq \delta$. The line segment $L$ belongs to $\Delta \setminus \Delta^\prime$. By the above, $\lim_{\tau \to \infty} \left(\#(A\Lambda_q \cap \tau \Delta) - \#(A\Lambda_q \cap \tau \Delta^\prime)\right)/\tau^2 = 2\left(\area(\Delta) - \area(\Delta^\prime)\right)/m_q(R_q) \leq 2\delta/m_q(R_q)$. It thus follows that for all $\epsilon > 0$, there exists $\tau_0 = \tau_0(A\Lambda_q, \Delta, \delta, \epsilon)$ such that for all $\tau > \tau_0$ we have
\begin{equation*}
\frac{\#\left(A\Lambda_q \cap \tau L\right)}{\tau^2} \leq \frac{\#(A\Lambda_q \cap \tau \Delta) - \#(A\Lambda_q \cap \tau \Delta^\prime)}{\tau^2} \leq \frac{2\delta}{m_q(R_q)} + \epsilon.
\end{equation*}
By the arbitrariness of $\delta$ and $\epsilon$, we get $\lim_{\tau \to \infty} \frac{\#\left(A\Lambda_q \cap \tau L\right)}{\tau^2} = 0$. This proves that adding or removing a finite number of line segments does not affect the main term for the asymptotic growth of the number of elements of $A\Lambda_q$ in homothetic dilations of triangles.
\end{proof}

\begin{proof}[Proof of \cref{corollary: equidistribution in the square}]
That the set functions $\lambda_\tau^{A\Lambda_q}$ are probability measures on $[-1, 1]^2$ is clear. We proceed to prove that they converge weakly to $\Unif_{[-1, 1]}$.

First, we note that given any rectangle $\mathcal{R}$ in the plane, we can express $\mathcal{R}$ using the union and/or difference of four triangles each having a vertex at the origin. From this follows that $\lim_{\tau \to \infty} \frac{\#(A\Lambda_q \cap \tau \mathcal{R})}{\tau^2} = \frac{2}{m_q(R_q)} \area(\mathcal{R})$. Consequently, if $\mathcal{R}$ belongs to $[-1,1]^2$, then $\lim_{\tau \to \infty} \lambda_\tau^{A\Lambda_q}(\mathcal{R}) = \Unif_{[-1, 1]^2}(\mathcal{R})$.

Fix a continuous function $f : [-1, 1]^2 \to \RR$. Given a $\delta > 0$, there exists a finite partition $\mathscr{P} = \mathscr{P}(f, \delta)$ of the square $[-1, 1]^2$ into rectangles such that the difference between the supremum and infimum of $f$ over each of the rectangles in the partition does not exceed $\delta$. That is, $\sup_\mathcal{R}(f) \leq \inf_\mathcal{R}(f) + \delta$ for all $\mathcal{R} \in \mathscr{P}$. (This is possible by the uniform continuity of $f$ over $[-1, 1]^2$.) Given $\epsilon > 0$, there exists $\tau_0 = \tau_0(A\Lambda_q, \mathscr{P}(f, \delta), \epsilon)$ such that $\left|\lambda_\tau^{A\Lambda_q}(\mathcal{R}) - \Unif_{[-1, 1]^2}(\mathcal{R})\right| \leq \epsilon$ for all $\tau > \tau_0$, and all $\mathcal{R} \in \mathscr{P}$. We thus have
\begin{eqnarray*}
\lambda_\tau^{A\Lambda_q}(f) &\leq& \sum_{\mathcal{R} \in \mathscr{P}} \sup_\mathcal{R}(f)\ \lambda_\tau^{A\Lambda_q}(\mathcal{R}) \\
 &\leq& \epsilon \max_{[-1,1]^2}(|f|) \#\mathscr{P} + \sum_{\mathcal{R} \in \mathscr{P}} \sup_{\mathcal{R}}(f) \Unif_{[-1, 1]^2}(\mathcal{R}) \\
 &\leq& \epsilon \max_{[-1,1]^2}(|f|) \#\mathscr{P} + \delta + \sum_{\mathcal{R} \in \mathscr{P}} \inf_\mathcal{R}(f) \Unif_{[-1, 1]^2}(\mathcal{R}) \\
 &\leq& \epsilon \max_{[-1,1]^2}(|f|) \#\mathscr{P}(f, \delta) + \delta + \Unif_{[-1, 1]^2}(f).
\end{eqnarray*}
Similarly
\begin{eqnarray*}
\lambda_\tau^{A\Lambda_q}(f) &\geq& \sum_{\mathcal{R} \in \mathscr{P}} \inf_\mathcal{R}(f)\ \lambda_\tau^{A\Lambda_q}(\mathcal{R}) \\
 &\geq& -\epsilon \max_{[-1,1]^2}(|f|) \#\mathscr{P} + \sum_{\mathcal{R} \in \mathscr{P}} \inf_\mathcal{R}(f)\Unif_{[-1,1]^2}(\mathcal{R}) \\
 &\geq& -\epsilon \max_{[-1,1]^2}(|f|) \#\mathscr{P} - \delta + \sum_{\mathcal{R} \in \mathscr{P}} \sup_{\mathcal{R}}(f)\Unif_{[-1,1]^2}(\mathcal{R}) \\
 &\geq& -\epsilon \max_{[-1,1]^2}(|f|) \#\mathscr{P}(f, \delta) - \delta + \Unif_{[-1, 1]^2}(f).
\end{eqnarray*}
That is, $\left|\lambda_\tau^{A\Lambda_q}(f) - \Unif_{[-1,1]^2}(f)\right| \leq \epsilon \max_{[-1,1]^2}(|f|) \#\mathscr{P}(f, \delta) + \delta$. By the arbitrariness of $\delta$ and $\epsilon$, we get $\lim_{\tau \to \infty} \lambda_\tau^{A\Lambda_q}(f) = \Unif_{[-1, 1]^2}(f)$. This proves the claim.
\end{proof}

\subsection{$G_q$-Farey Statistics}
\label{subsection: G_q Farey statistics}

In \cref{proposition: limiting distributions of Farey statistics} below we derive the limiting distribution of quantities that can be expressed as functions in the $G_q$-Farey triangle representatives (\cref{subsection: G_q next term algorithm}) of the elements of the sets $\mathcal{F}_I(A\Lambda_q, \tau)$ (\cref{subsection: limiting distributions of FTR}) as $\tau \to \infty$. As examples of said distributions, we consider the slope gap distribution of $\Lambda_q$ in \cref{subsubsection: slopegap distribution}, and the distribution of the Euclidean distance between $G_q$-Ford circles in \cref{subsubsection: centdist distribution}.

\begin{proposition}
\label{proposition: limiting distributions of Farey statistics}
Let $F : \mathscr{T}^q \to \RR$ be a function, continuous on the Farey triangle $\mathscr{T}^q$ except perhaps on the image of finitely many $C^1$ curves $\{c_i : I_i \to \mathscr{T}^q\}_{i=1}^m$, with $I_i \subset \RR$ being finite, closed intervals of $\RR$. For any $A \in \SL(2, \RR)$, and any finite interval $I \subset \RR$, the limit of the distribution
\begin{equation*}
\FareyStat_{F,\tau}^{A\Lambda_q,I}(t) := \frac{\# \left\{\mathbf{u} \in \mathcal{F}_I(A\Lambda_q, \tau) \mid F\left(\FTR_q(A \Lambda_q, \tau, \mathbf{u})\right) \geq t\right\}}{N_I(A \Lambda_q, \tau)}
\end{equation*}
as $\tau \to \infty$ exists for all $t \in \RR$, and is equal to
\begin{equation*}
\FareyStat_F(t) = m_q\left(\mathbbm{1}_{F \geq t}\right),
\end{equation*}
where $\mathbbm{1}_{F \geq t}$ is the indicator function of the subset
\begin{equation*}
\left\{(a, b) \in \mathscr{T}^q \mid F(a,b) \geq t\right\}
\end{equation*}
of $\mathscr{T}^q$, and $m_q$ is the Lebesgue probability measure $dm_q = \frac{2}{\lambda_q} dadb$ on $\mathscr{T}^q$.
\end{proposition}

\begin{proof}
Fix $t \in \RR$. We then have
\begin{eqnarray*}
\FareyStat_{F,\tau}^{A\Lambda_q,I}(t)
 &=& \frac{\# \left\{\mathbf{u} \in \mathcal{F}_I(A \Lambda_q, \tau) \mid F \left(\FTR_q(A, \tau, \mathbf{u})\right) \geq t\right\}}{N_I(A\Lambda_q, \tau)} \\
 &=& \frac{1}{N_I(A \Lambda_q, \tau)} \sum_{i=0}^{N_I(A \Lambda_q, \tau) - 1} \mathbbm{1}_{F \geq t}\left(\FTR_q(A, \tau, \mathbf{u}_i\right) \\
 &=& \rho_{A \Lambda_q, I, \tau}\left(\mathbbm{1}_{F \geq t}\right),
\end{eqnarray*}
and so we proceed to show that $\lim_{\tau \to \infty} \rho_{A \Lambda_q, I, \tau} \left(\mathbbm{1}_{F \geq t}\right) = m_q\left(\mathbbm{1}_{F \geq t}\right)$.

Consider the following sets
\begin{eqnarray*}
A_t &=& \{(a, b) \in \mathscr{T}^q \mid F(a,b) \geq t\}, \\
B_t &=& \{(a, b) \in \mathscr{T}^q \mid F(a, b) \geq t\} \cup \bigcup_{i=1}^m c_i(I_i), \text{ and} \\
C &=& \bigcup_{i=1}^m c_i(I_i).
\end{eqnarray*}
The set $C_t$ is null with respect to the measure $m_q$, and $A_t \Delta B_t \subseteq C$, and so $m_q(\mathbbm{1}_{A_t}) = m_q(\mathbbm{1}_{B_t})$. The sets $B_t$ and $C$ are closed, and so their indicator functions $\mathbbm{1}_{B_t}$ and $\mathbbm{1}_C$ are bounded, and upper semi-continuous. Theorem \ref{theorem: weak limit and asymptotic growth} gives $\lim_{\tau \to \infty} \rho_{A\Lambda_q, I, \tau}(\mathbbm{1}_{B_t}) = m_q(\mathbbm{1}_{B_t})$, and $\lim_{\tau \to \infty} \rho_{A\Lambda_q, I, \tau}(\mathbbm{1}_{C_t}) = m_q(\mathbbm{1}_{C_t}) = 0$. Since $|\mathbbm{1}_{A_t} - \mathbbm{1}_{B_t}| \leq \mathbbm{1}_C$ on all of $\mathscr{T}^q$, we get $\lim_{\tau \to \infty} \rho_{A\Lambda_q, I, \tau}(\mathbbm{1}_{A_t}) = m_q(\mathbbm{1}_{A_t})$.
\end{proof}

\subsubsection{Slope Gap Distribution}
\label{subsubsection: slopegap distribution}

Let $A \in \SL(2, \RR)$, and $\tau > 0$ be such that $\mathcal{F}(A\Lambda_q, \tau) \neq \emptyset$. Given two vectors $\mathbf{u}_0, \mathbf{u}_1 \in \mathcal{F}(A\Lambda_q, \tau)$ with consecutive slopes, we denote the difference between the slopes of $\mathbf{u}_0$ and $\mathbf{u}_1$ by $\slopegap(A\Lambda_q, \tau, \mathbf{u}_0) = \slope(\mathbf{u}_1) - \slope(\mathbf{u}_0)$. We have the following on the limiting distribution of $\slopegap$.

\begin{corollary}
\label{corollary: limiting distribution of slopegap}
Let $A \in \SL(2, \RR)$, $I \subset \RR$ be a finite interval. The limit of
\begin{equation*}
\SlopeGap_\tau^{A\Lambda_q, I}(t) := \frac{\# \left\{\mathbf{u} \in \mathcal{F}_I(A\Lambda_q, \tau) \mid \tau^2 \slopegap(A\Lambda_q, \tau, \mathbf{u}) \geq t\right\}}{N_I(A\Lambda_q, \tau)}
\end{equation*}
as $\tau \to \infty$ exists for all $t \in \RR$, and is equal to $m_q(\mathbbm{1}_{R_q \geq t})$, where $m_q = \frac{2}{\lambda_q} da db$ is the Lebesgue probability measure on the $G_q$-Farey triangle $\mathscr{T}^q$.
\end{corollary}

\begin{proof}
Let $\tau > 0$ be such that $\mathcal{F}_I(A\Lambda_q, \tau) = \{\mathbf{u}_n = (q_n, a_n)^T\}_{n=0}^{N_I(A\Lambda_q, \tau) - 1} \neq \emptyset$. For any $0 \leq n \leq N_I(A\Lambda_q, \tau) - 2$, we have by \cref{theorem: next term algorithm} that
\begin{equation*}
\slopegap(A\Lambda_q, \tau, \mathbf{u}_n) = \frac{a_{n+1}}{q_{n+1}} - \frac{a_n}{q_n} = \frac{1}{\tau^2} R_q(\FTR_q(A\Lambda_q, \tau, \mathbf{u}_n)).
\end{equation*}
This implies that $\SlopeGap_\tau^{A\Lambda_q,I}(t) = \FareyStat_{R_q,\tau}^{A\Lambda_q,I}(t)$, and the proposition then follows from \cref{proposition: limiting distributions of Farey statistics}.
\end{proof}

\subsubsection{The $G_q$-Ford Circles, and Their Geometric Statistics}
\label{subsubsection: centdist distribution}

\begin{figure}
\centering
\includegraphics[scale=0.45]{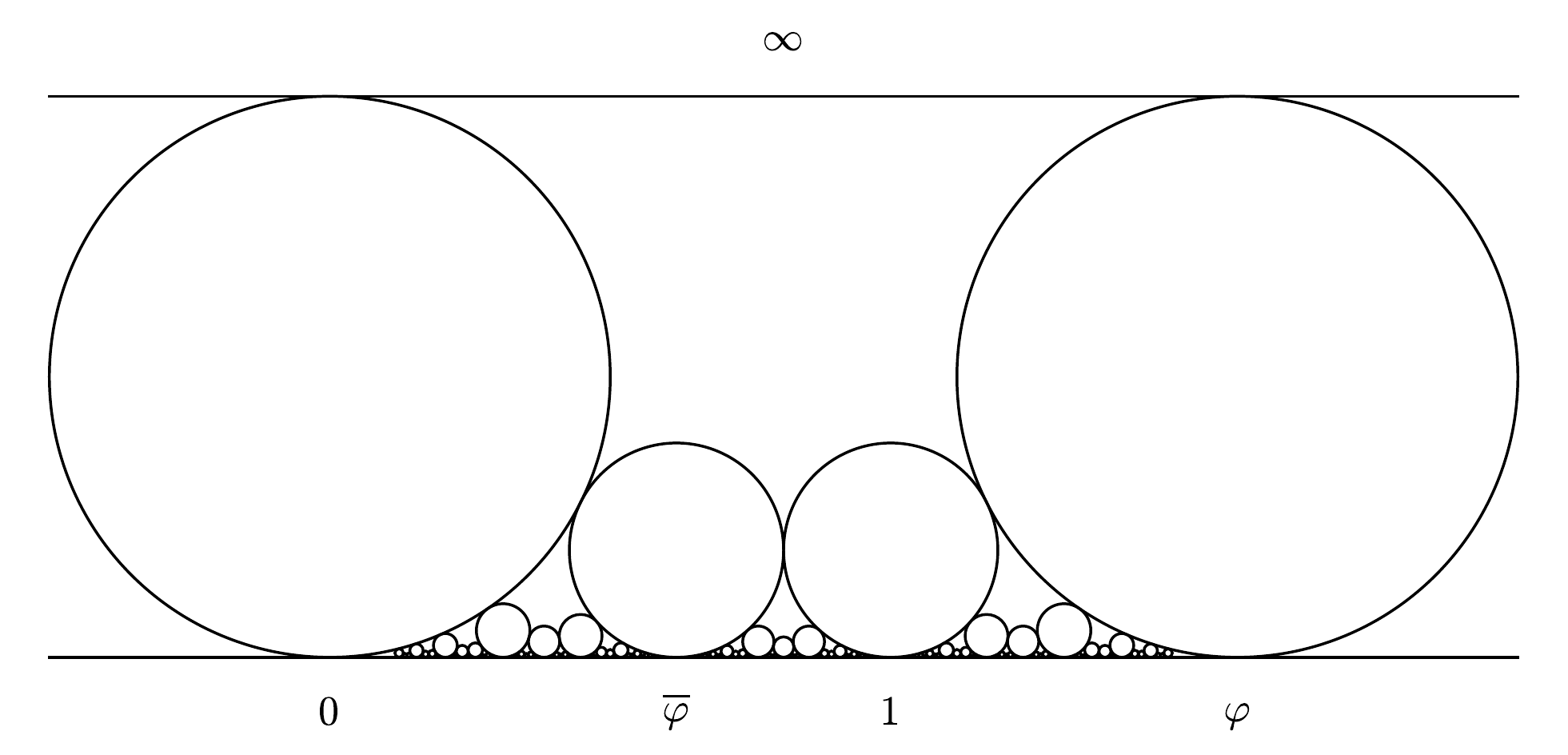}
\caption{The $G_5$-Ford circles corresponding to the vectors in $\mathcal{F}_{[0, \varphi]}(\Lambda_5, \infty)$ along with the circle at infinity. The circles tangent to the line $y=0$ at $\overline{\varphi}$, $1$, and $\varphi$ are the $G_5$-Stern-Brocot children of the circles at $0$ and infinity.}
\label{figure: G_5 Ford circles}
\end{figure}

For any point $\mathbf{w} = (r,s)^T \in \RR^2$, the \emph{Ford circle} $C[\mathbf{w}]$ \cite{Ford1938-ya} corresponding to $\mathbf{w}$ is defined to be either
\begin{itemize}
\item the circle with radius $\frac{1}{2r^2}$, and center at $\left(\frac{s}{r}, \frac{1}{2r^2}\right)$, if $r \neq 0$, or
\item the straight line $y = s^2$, if $r = 0$.
\end{itemize}
It is well-known that for any two vectors $\mathbf{w}_1, \mathbf{w}_2 \in \RR^2$, the Ford circles $C[\mathbf{w}_1]$ and $C[\mathbf{w}_2]$ intersect if $|\mathbf{w}_0 \wedge \mathbf{w}_1| < 1$, are tangent if $|\mathbf{w}_0 \wedge \mathbf{w}_1| = 1$, and are wholly external if $|\mathbf{w}_0 \wedge \mathbf{w}_1| > 1$.

It follows from \cref{theorem: G_q Stern Brocot process is well-defined and exhaustive} and \cref{corollary: odds and ends} that for any $A \in \SL(2, \RR)$, the Ford circles corresponding to any two distinct elements of $A\Lambda_q \cap S_\infty$ are either tangent or wholly external, and that the $G_q$-Stern-Brocot children of any two unimodular vectors of $A\Lambda_q \cap S_\infty$ correspond to a chain of $q - 2$ tangent circles between the two circles corresponding to the ``parents''.

Let $A \in SL(2, \RR)$, and $\tau > 0$ be such that $\mathcal{F}(A \Lambda_q, \tau) \neq \emptyset$. Given two vectors $\mathbf{u}_0, \mathbf{u}_1 \in \mathcal{F}(A \Lambda_q, \tau)$ with consecutive slopes, we denote the distance between the centers of $C[\mathbf{u}_0]$ and $C[\mathbf{u}_1]$ by $\centdist(A \Lambda_q, \tau, \mathbf{u}_0)$. We have the following on the limiting distribution of $\centdist$, extending a result from \cite{Athreya2015-gk} for $G_3$-Ford circles.

\begin{corollary}
\label{corollary: limiting distribution of centdist}
Let $A \in \SL(2, \RR)$, and $I \subset \RR$ be a finite interval. The limit of
\begin{equation*}
\CentDist_\tau^{A\Lambda_q,I}(t) := \frac{\# \left\{\mathbf{u} \in \mathcal{F}_I(A\Lambda_q, \tau) \mid \tau^2 \centdist(A\Lambda_q, \tau, \mathbf{u}) \geq t\right\}}{N_I(A\Lambda_q, \tau)}
\end{equation*}
as $\tau \to \infty$ exists for all $t \in \RR$, and is equal to $m_q(\mathbbm{1}_{F_q \geq t})$, where $m_q = \frac{2}{\lambda_q}dadb$ is the Lebesgue probability measure on the $G_q$-Farey triangle $\mathscr{T}^q$, and $F_q : \mathscr{T}^q \to \RR$ is the function defined by
\begin{equation*}
F_q(a, b) = \sqrt{R_q(a, b)^2 + \frac{1}{4} \left(\frac{1}{L_1^q(a, b)^2} - \frac{1}{L_0^q(a, b)^2}\right)^2},
\end{equation*}
where $L_0^q$ and $L_1^q$ are as in \cref{theorem: next term algorithm}.
\end{corollary}

As an immediate consequence of the second claim in \cref{corollary: odds and ends}, we get the following weak form of Dirichelet's approximation theorem for $\Lambda_q$.

\begin{proposition}
\label{proposition: weak Dirichelet approximation}
Let $A \in \SL(2, \RR)$, and $\alpha \in \RR$. The line $x = \alpha$ either passes through the center of a $G_q$-Ford circle corresponding to a vector in $A \Lambda_q$, or there exist infinitely many vectors in $A \Lambda_q$ whose Ford circles intersect $x = \alpha$. In particular, $\alpha$ is either the slope of a vector in $A \Lambda_q$, or there exist infinitely many $(q, a)^T \in A \Lambda_q$ such that
\begin{equation*}
\left|\alpha - \frac{a}{q}\right| \leq \frac{1}{2q^2}.
\end{equation*}
\end{proposition}

\begin{figure}
\label{figure: Euclidean distance distribution}
\centering
\includegraphics[scale=0.75]{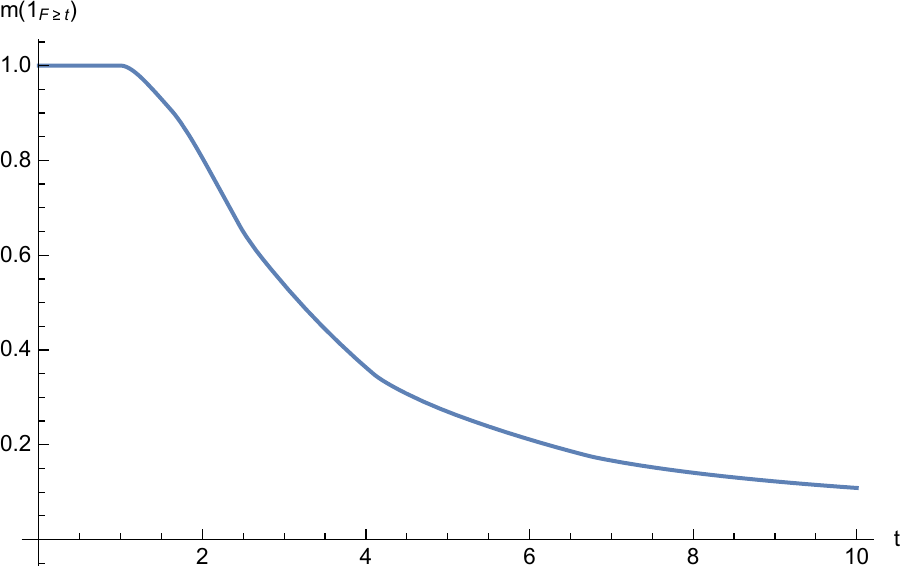}
\caption{The graph of the limiting distribution $\lim_{t \to \infty} \CentDist_\tau^{\Lambda_5,I}(t) = m_5(\mathbbm{1}_{F_5 \geq t})$ of the Euclidean distance between successive $G_5$-Ford circles from \cref{corollary: limiting distribution of centdist}.}
\end{figure}

\begin{proof}
Let $\tau > 0$ be such that $\mathcal{F}_I(A\Lambda_q, \tau) = \{\mathbf{u}_n = (q_n, a_n)^T\}_{n=0}^{N_I(A\Lambda_q, \tau) - 1} \neq \emptyset$. For any $0 \leq n \leq N_I(A\Lambda_q, \tau) - 2$, we have by \cref{theorem: next term algorithm} that
\begin{equation*}
\frac{a_{n+1}}{q_{n+1}} - \frac{a_n}{q_n} = \frac{1}{\tau^2} R_q(\FTR_q(A\Lambda_q, \tau, \mathbf{u}_n)),
\end{equation*}
and
\begin{equation*}
\frac{1}{2q_{n+1}^2} - \frac{1}{2q_n^2} = \frac{1}{2 \tau^2 L_1^q(\FTR_q(A\Lambda_q, \tau, \mathbf{u}_n))^2} - \frac{1}{2 \tau^2 L_0^q(\FTR_q(A\Lambda_q, \tau, \mathbf{u}_n))^2}.
\end{equation*}
From this follows that the distance between the centers of $C[\mathbf{u}_n]$ and $C[\mathbf{u}_{n+1}]$ is given by
\begin{eqnarray*}
\centdist(A\Lambda_q, \tau, \mathbf{u}_n) &=& \sqrt{\left(\frac{a_{n+1}}{q_{n+1}} - \frac{a_n}{q_n}\right)^2 + \left(\frac{1}{2q_{n+1}^2} - \frac{1}{2q_n^2}\right)^2} \\
 &=& \frac{1}{\tau^2} F_q(\FTR_q(A\Lambda_q, \tau, \mathbf{u}_n)).
\end{eqnarray*}
This implies that $\CentDist_\tau^{A\Lambda_q,I}(t) = \FareyStat_{F_q,\tau}^{A\Lambda_q,I}(t)$, and the proposition then follows from \cref{proposition: limiting distributions of Farey statistics}.
\end{proof}


\begin{thebibliography}{9}

\bibitem{Athreya2015-nq} Athreya, J. S., Chaika, J., \& Lelievre, S. (2015). The gap distribution of slopes on the golden L. Recent trends in ergodic theory and dynamical systems, 631, 47--62.

\bibitem{Athreya2015-gk} Athreya, J., Chaubey, S., Malik, A., \& Zaharescu, A. (2015). Geometry of Farey--Ford polygons. New York Journal of Mathematics, 21, 637--656.

\bibitem{Athreya2013-ql} Athreya, J. S., \& Cheung, Y. (2013). A Poincar\'{e} Section for the Horocycle Flow on the Space of Lattices. International Mathematics Research Notices.

\bibitem{Augustin2001-tl} Augustin, V., Boca, F. P., Cobeli, C., \& Zaharescu, A. (2001). The h-spacing distribution between Farey points. Mathematical Proceedings of the Cambridge Philosophical Society, 131(1), 23–38.

\bibitem{Boca2008-bu} Boca, F. P. (2008). A Problem of Erd\"{o}s, Sz\"{o}z, and Tur\'{a}n Concerning Diophantine Approximations. International Journal of Number Theory, 04(04), 691--708.

\bibitem{Boca2001-he} Boca, F. P., Cobeli, C., \& Zaharescu, A. (2001). A conjecture of R. R. Hall on Farey points. Journal fur die Reine und Angewandte Mathematik, 2001(535).

\bibitem{Dani1984-ji} Dani, S. G., \& Smillie, J. (1984). Uniform distribution of horocycle orbits for Fuchsian groups. Duke Mathematical Journal, 51(1), 185--194.

\bibitem{Davis2018-al} Davis, D., \& Lelievre, S. (2018, October 26). Periodic paths on the pentagon, double pentagon and golden L. arXiv [math.DS]. http://arxiv.org/abs/1810.11310

\bibitem{Erdos1958-xi} Erd\"{o}s, P., Sz\"{u}sz, P., \& Tur\'{a}n, P. (1958). Remarks on the theory of diophantine approximation. Colloquium Mathematicum, 6(1), 119--126.

\bibitem{Fisher2014-ke} Fisher, A. M., \& Schmidt, T. A. (2014). Distribution of approximants and geodesic flows. Ergodic Theory and Dynamical Systems, 34(6), 1832--1848.

\bibitem{Ford1938-ya} Ford, L. R. (1938). Fractions. The American mathematical monthly: the official journal of the Mathematical Association of America, 45(9), 586--601.

\bibitem{Hall2003-it} Hall, R. R., \& Shiu, P. (2003). The index of a Farey sequence. Michigan Mathematical Journal, 51(1), 209–223.

\bibitem{Hardy1979-np} Hardy, G. H., Wright, E. M., \& (Edward Maitland), E. (1979). An Introduction to the Theory of Numbers. Clarendon Press.

\bibitem{Heersink2016-hg} Heersink, B. (2016). Poincar\'{e} sections for the horocycle flow in covers of $\SL(2,\RR)/\SL(2,\ZZ)$ and applications to Farey fraction statistics. Monatshefte f\"{u}r Mathematik, 179(3), 389--420.

\bibitem{Kleinbock1996-ag} Kleinbock, D. Y., \& Margulis, G. A. (1996). Bounded orbits of nonquasiunipotent flows on homogeneous spaces. American Mathematical Society Translations, Series 2, 141–172.

\bibitem{Lang2016-qs} Lang, C. L., \& Lang, M. L. (2016). Arithmetic and geometry of the Hecke groups. Journal of Algebra, 460, 392--417.

\bibitem{Panti2015-jz} Panti, G. (2015, March 9). The weighted Farey sequence and a sliding section for the horocycle flow. arXiv [math.DS]. http://arxiv.org/abs/1503.02539

\bibitem{Uyanik2016-gx} Uyanik, C., \& Work, G. (2016). The Distribution of Gaps for Saddle Connections on the Octagon. International Mathematics Research Notices, 2016(18), 5569--5602.

\bibitem{Smillie2010-nb} Smillie, J., \& Weiss, B. (2010). Characterizations of lattice surfaces. Inventiones Mathematicae, 180(3), 535--557.

\bibitem{Xiong2006-ef} Xiong, M. S., \& Zaharescu, A. (2006). A problem of Erd\"{o}s-Sz\"{u}sz-Tur\'{a}n on diophantine approximation. Acta Arithmetica, 125(2), 163.

\end{thebibliography}
\end{document}